\documentclass[12pt,reqno]{amsart}
\usepackage{amsmath}
\usepackage{amssymb}
\usepackage{amstext}
\usepackage{a4wide}
\usepackage{graphicx}
\allowdisplaybreaks \numberwithin{equation}{section}
\usepackage{color}
\usepackage{cases}

\numberwithin{equation}{section}

\newtheorem{theorem}{Theorem}[section]
\newtheorem{proposition}[theorem]{Proposition}

\newtheorem{lemma}[theorem]{Lemma}

\theoremstyle{definition}

\newtheorem{innercustomthm}{Lemma}
\newenvironment{customthm}[1]
{\renewcommand\theinnercustomthm{#1}\innercustomthm}
{\endinnercustomthm}

\newtheorem{definition}[theorem]{Definition}

\theoremstyle{remark}
\newtheorem{remark}[theorem]{Remark}

\begin{document}

\title[Helical vortices with small cross-section for 3D incompressible Euler equation]
{Helical vortices with small cross-section for 3D incompressible Euler equation}

\author{Daomin Cao, Jie Wan}
	
\address{Institute of Applied Mathematics, Chinese Academy of Sciences, Beijing 100190, and University of Chinese Academy of Sciences, Beijing 100049,  P.R. China}
\email{dmcao@amt.ac.cn}
\address{School of Mathematics and Statistics, Beijing Institute of Technology, Beijing 100081,  P.R. China}
\email{wanjie@bit.edu.cn}


\begin{abstract}
In this article, we construct traveling-rotating helical vortices with small cross-section to the 3D incompressible Euler equations in an infinite pipe, which tend asymptotically to singular helical vortex filament evolved by the binormal curvature flow. The construction is based on studying a general semilinear elliptic problem in divergence form
\begin{equation*}
\begin{cases}
-\varepsilon^2\text{div}(K(x)\nabla u)=  (u-q|\ln\varepsilon|)^{p}_+,\ \ &x\in \Omega,\\
u=0,\ \ &x\in\partial \Omega,
\end{cases}
\end{equation*}
for small values of $ \varepsilon. $ Helical vortex solutions concentrating near several helical filaments with polygonal symmetry are also constructed.

\end{abstract}

\maketitle

\section{Introduction and main results}
The movement of an ideal incompressible  flow  confined in a 3D  domain $ D $ is governed by the following Euler equation
\begin{equation}\label{Euler eq}
\begin{cases}
\partial_t\mathbf{v}+(\mathbf{v}\cdot \nabla)\mathbf{v}=-\nabla P,\ \ &D\times (0,T),\\
\nabla\cdot \mathbf{v}=0,\ \ &D\times (0,T),\\
\mathbf{v}(\cdot, 0)=\mathbf{v}_0(\cdot), \ \ &D,
\end{cases}
\end{equation}
where the domain $ D\subseteq \mathbb{R}^3 $, $ \mathbf{v}=(v_1,v_2,v_3) $ is the velocity field, $ P$ is the scalar pressure and  $ \mathbf{v}_0 $ is the initial velocity field.  If $ D $ has a boundary,  the following impermeable boundary condition
is usually assumed
\begin{equation*}
\mathbf{v}\cdot \mathbf{n}=0,\ \ \partial D\times (0,T),
\end{equation*}
where $ \mathbf{n} $ is the outward unit normal to $ \partial D $.

Define the associated vorticity field   $ \mathbf{w}=(w_1,w_2,w_3)=curl \mathbf{v}=\nabla\times \mathbf{v} $, which describes the rotation of the fluid.  Then $ \mathbf{w} $ satisfies the vorticity equations
\begin{equation}\label{Euler eq2}
\begin{cases}
\partial_t\mathbf{w}+(\mathbf{v}\cdot \nabla)\mathbf{w}=(\mathbf{w}\cdot \nabla)\mathbf{v},\ \ &D\times (0,T),\\
\mathbf{w}(\cdot, 0)=\nabla\times \mathbf{v}_0(\cdot), \ \ &D.
\end{cases}
\end{equation}
For background of the 3D incompressible Euler equation, see the classical literature \cite{MB, MP}.

 Helmholtz \cite{He} began the study of Euler equation in 1858, who first considered the vorticity equations of the flow and found that the vortex rings, which are toroidal regions in which the vorticity has small cross-section, translate with a constant speed alone the axis of symmetry. The translating speed of vortex rings was then studied by Kelvin and Hick \cite{Lamb} in 1899. Define  the circulation of a vortex
\begin{equation}\label{1001}
c=\oint_l\mathbf{v}\cdot \mathbf{t}dl=\iint_{\sigma}\mathbf{w}\cdot\mathbf{n}d\sigma,
\end{equation}
where $ l $ is any oriented curve with tangent vector field $ \mathbf{t} $ that
encircles the vorticity region once and $ \sigma $ is any surface with boundary $ l $.
 \cite{Lamb} showed that if the vortex ring has radius $ r^* $, circulation $ c $ and its cross-section $ \varepsilon $ is small, then the vortex
ring moves at the velocity
\begin{equation}\label{1002}
\frac{c}{4\pi r^*}\left( \ln\frac{8r^*}{\varepsilon}-\frac{1}{4}\right).
\end{equation}
Then Da Rios \cite{DR} in 1906, and Levi-Civita \cite{LC} in 1908, formally found the general law of motion of a vortex filament with a small section of radius $ \varepsilon $ and a fixed circulation, uniformly distributed around an evolving curve $ \Gamma(t) $, which is well-known as the binormal curvature flow, or the  localized induction approximation (LIA). Roughly speaking, under suitable assumptions on the solution, the curve evolves by the  binormal  flow, with a large velocity of order $ | \ln\varepsilon| $. More precisely, if $ \Gamma(t) $ is parameterized as $ \gamma(s, t) $, where $ s $ is the  parameter of arclength, then $ \gamma(s, t) $ asymptotically obeys a law of the form (see \cite{LC2}, p. 30, Eq. (8') and \cite{Ric2} p. 260 Eq. (62))
\begin{equation}\label{1003}
\partial_t \gamma=\frac{c}{4\pi}|\ln\varepsilon|(\partial_s\gamma\times\partial_{ss}\gamma)=\frac{c\bar{K}}{4\pi}|\ln\varepsilon|\mathbf{b}_{\gamma(t)},
\end{equation}
where $ c $ is the circulation of the velocity field on the boundary of sections
to the filament, which is assumed to be a constant independent of $ \varepsilon $, $ \mathbf{b}_{\gamma(t)}  $ is the binormal unit vector and $ \bar{K} $ is its local curvature. If we scale $ t = | \ln\varepsilon|^{-1}\tau $, then
\begin{equation}\label{1004}
\partial_\tau \gamma=\frac{c\bar{K}}{4\pi}\mathbf{b}_{\gamma(\tau)}.
\end{equation}
Hence, under LIA vortex filaments
move simply in the binormal direction with speed proportional to the local curvature and the circulation. It is worthwhile to note that, when $ \Gamma $ is a circular filament, the leading term of \eqref{1002} coincides with the coefficient of right hand side of \eqref{1003} since in this case the local curvature $ \bar{K}=\frac{1}{r^*}. $ The localized induction approximation found by Da Rios was applied to various physical problems, for instance, the induction due to electric currents in a wire \cite{DR2}, the gravitational effect associated with Saturnian rings \cite{LC3} and vortex motion \cite{LC2}, for more detail, see the survey papers by Ricca \cite{Ric,Ric2}.

From mathematical justification, Jerrad and Seis \cite{JS} first gave a precise form to Da Rios’ computation under some mild conditions on a solution to \eqref{Euler eq2} which remains suitably concentrated
around an evolving vortex filament. Their result shows that under some conditions of a solution $ \mathbf{w}_\varepsilon $ of \eqref{Euler eq2}, there holds in the sense of distribution,
\begin{equation}\label{1005}
\mathbf{w}_\varepsilon(\cdot,|\ln\varepsilon|^{-1}\tau)\to c \delta_{\gamma(\tau)}\mathbf{t}_{\gamma(\tau)},\ \ \text{as}\ \varepsilon\to0,
\end{equation}
 where $ \gamma(\tau) $ satisfies \eqref{1004}, $ \mathbf{t}_{\gamma(\tau)} $ is the tangent unit vector of $ \gamma $ and $ \delta_{\gamma(\tau)} $ is the uniform Dirac measure on the curve. See \cite{JS2} for more results of this problem.

Until now the existence of a family of solutions to \eqref{Euler eq2} satisfying \eqref{1005}, where $ \gamma(\tau) $ is a given curve  evolved by the binormal flow \eqref{1004}, is still an open problem. This problem is well-known as the vortex filament conjecture, which is unsolved except for the filament being several kinds of special curves: the straight lines, the traveling circles and the traveling-rotating helices.
For the problem of vortex concentrating near straight lines, it corresponds to the planar Euler equations concentrating near a collection of given points governed by the 2D point vortex model, see \cite{CF,CLW,DDMW,Li,MP,SV} for example. When the filament is a traveling circle with radius $ r^* $, by \eqref{1004} the curve is
\begin{equation}\label{1006}
\gamma(s,\tau)=\left( r^*\cos\left( \frac{s}{r^*}\right), r^*\sin\left( \frac{s}{r^*}\right), \frac{c}{4\pi r^*}\tau\right)^t,
\end{equation}
where $ \mathbf{v}^t $ is the transposition of a vector $ \mathbf{v} $. Fraenkel \cite{Fr1} first gave a construction of  vortex rings with small cross-section without change of form concentrating near a traveling circle satisfying \eqref{1006} in sense of \eqref{1005} and then many articles showed the desingularization results under a variety of conditions, such as constructing vortex rings in different kinds of domains with different vortex profiles, see \cite{ALW, Bur, DV, FB} for instance.

For vortex filament being a helix satisfying \eqref{1004},  the curve is parameterized as
\begin{equation}\label{1007}
\gamma(s,\tau)=\left( r_*\cos\left( \frac{-s-a_1\tau}{\sqrt{k^2+r_*^2}}\right),  r_*\sin\left( \frac{-s-a_1\tau}{\sqrt{k^2+r_*^2}}\right), \frac{ks-b_1\tau}{\sqrt{k^2+r_*^2}}\right)^t,
\end{equation}
where $ r_* > 0, k\neq 0 $ are constants characterizing  the distance between a point in $ \gamma(\tau) $ and the $ x_3 $-axis and the pitch of the helix, and
\begin{equation*}
a_1=\frac{ck}{4\pi (k^2+r_*^2)}, \ b_1=\frac{cr_*^2}{4\pi (k^2+r_*^2)}.
\end{equation*}
Note that the local curvature and torsion of the helix are $ \frac{r_*}{k^2+r_*^2} $ and $ \frac{k}{k^2+r_*^2} $ respectively, and the parametrization \eqref{1007} satisfies \eqref{1004}. It   should be noted that the curve parameterized by \eqref{1007} is a traveling-rotating helix. Let us define  for any $ \theta\in[0,2\pi] $
\begin{equation*}
\bar{R}_\theta=\begin{pmatrix}
\cos\theta & \sin\theta  \\
-\sin\theta &\cos\theta
\end{pmatrix}, \  \
\bar{Q}_\theta=\begin{pmatrix}
\bar{R}_\theta & 0  \\
0 & 1
\end{pmatrix}.
\end{equation*}
One computes directly that
\begin{equation*}
\gamma(s,\tau)=\bar{Q}_{\frac{a_1\tau}{\sqrt{k^2+r_*^2}}}\gamma(s,0)+\left( 0,0,-\frac{b_1\tau}{\sqrt{k^2+r_*^2}}\right)^t.
\end{equation*}
We can readily check that \eqref{1007} with $ k>0 $ and $ k<0 $ correspond to the left-handed helix and the right-handed helix  respectively (for consistency, throughout this paper we always choose  a right-handed Cartesian reference centered at the origin). The problem of global well-posedness of solutions to the vorticity equation \eqref{Euler eq2} with helical symmetry was studied in many articles, see \cite{AS,BLN,Du,Jiu} for instance. For a helix $ \gamma(\tau) $ satisfying \eqref{1007}, there are a few results of existence of true solutions of \eqref{Euler eq2} concentrating on this curve in sense of \eqref{1005}. The only result is shown by D$\acute{\text{a}}$vila et al. \cite{DDMW2}, who considered traveling-rotating invariant Euler flows with right-handed helical symmetry concentrating near a single  helix and multiple helices in the whole space $ \mathbb{R}^3. $ In their work, by considering
\begin{equation*}
\begin{split}
-\text{div}(K_H(x)\nabla u)=f_\varepsilon\left( u-\alpha|\ln\varepsilon|\frac{|x|^2}{2}\right) \ \ \text{in}\ \ \mathbb{R}^2,
\end{split}
\end{equation*}
where $ K_H $ is an elliptic operator in divergence form (defined by \eqref{coef matrix}), $ f_\varepsilon(t)=\varepsilon^2e^t $ and $ \alpha $ is chosen properly, the authors construct solutions concentrating near a helix in the distributional sense.
Note that by the choice of  $ f_\varepsilon $, the support set of vorticity is still the whole plane.

The aim of this paper is to construct traveling-rotating solutions to Euler equations \eqref{Euler eq2} with helical symmetry in an infinite pipe, such that the support set of vortex is a helical tube with small cross-section $ \varepsilon $ without change of form, which tends asymptotically to a traveling-rotating helix  \eqref{1007} in sense of \eqref{1005}. To get these result, we study the existence and asymptotic behavior of solutions to a general semilinear elliptic equations (see \eqref{eq1}).
It should be noted that, Euler equations with helical symmetry can be regarded as the general case of 2D and 3D axisymmetric Euler equations. The cases $ k\to+\infty $ and $ k=0 $ correspond to the 2D Euler equations and 3D axisymmetric Euler equations, respectively. In contrast to the 2D and 3D axisymmetric problems, the associated operator $ \mathcal{L}_H $ in vorticity equations (see \eqref{vor str eq}) is an elliptic operator in divergence form, which can bring essential difficulty in the construction of solutions. First, it seems impossible to reduce the second-order operator $ \mathcal{L}_H $ to the standard Laplace operator by means of a single change of coordinates. Second, lack of understanding the properties of the Green's function of a general elliptic operator in divergence form is also a challenge. Moreover, since the eigenvalues of $ K_H $ are different, solutions of the associated limiting equations are not radially symmetric functions, which is totally different from the 2D and 3D axisymmetric cases.

To state our results, we need to introduce some notations first. Since the helical vortices we are to construct is translating-rotating symmetric, domains of the flow must be   helical domains with  rotating symmetry about $ x_3 $ axis, which are  the whole space and  infinite pipes with circular cross section. For any $ R^*>0 $, define $ B_{R^*}(0)\times \mathbb{R}=\{(x_1,x_2,x_3)\mid (x_2,x_2)\in B_{R^*}(0), x_3\in \mathbb{R}\} $ an infinite pipe in $ \mathbb{R}^3 $ whose section is a disc with radius $ R^* $. For two sets $ A,B $, define $ dist(A,B)=\min_{x\in A,y\in B}|x-y| $ the distance between sets $ A$ and $B $ and $ diam(A) $ the diameter of the set $ A $.

Our first result is concerned with the desingularization of traveling-rotating helical vortices  in $ B_{R^*}(0)\times \mathbb{R} $, whose support set has small cross-section $ \varepsilon $ and concentrates near a single left-handed helix  \eqref{1007} in sense of \eqref{1005}.

\begin{theorem}\label{thm01}
Let $ k>0 $, $ c>0 $ and $ r_*\in (0,R^*) $ be any given numbers. Let  $ \gamma(\tau) $ be the helix parameterized by equation \eqref{1007}. Then for any $ \varepsilon\in (0,\varepsilon_0] $ for some $ \varepsilon_0>0 $, there exists
a classical solution pair $ (\mathbf{v}_\varepsilon, P_\varepsilon)(x,t)\in C^1(B_{R^*}(0)\times \mathbb{R}\times \mathbb{R}^+) $ of  \eqref{Euler eq} such that the support set of  $ \mathbf{w}_\varepsilon $ is a topological traveling-rotating helical tube that does not change form and concentrates near the helix in sense of \eqref{1005}, that is for all $ \tau $,
\begin{equation*}
\mathbf{w}_\varepsilon(\cdot,|\ln\varepsilon|^{-1}\tau)\to c \delta_{\gamma(\tau)}\mathbf{t}_{\gamma(\tau)},\ \ \text{as}\ \varepsilon\to0.
\end{equation*}
Moreover, one has
\begin{enumerate}
	\item $ \mathbf{v}_\varepsilon\cdot \mathbf{n}=0$ on $\partial B_{R^*}(0)\times \mathbb{R}. $
	\item Define $ A_\varepsilon=supp(\mathbf{w}_\varepsilon)\cap \mathbb{R}^2\times\{0\} $ the cross-section of $ \mathbf{w}_\varepsilon $. Then there are $ R_1,R_2>0 $ such that
	\begin{equation*}
	R_1\varepsilon\leq diam(A_\varepsilon)\leq R_2\varepsilon.
	\end{equation*}
\end{enumerate}
\end{theorem}

\begin{remark}
	By the physical meaning of $ k $, the sign of $ k $ determines two different helical structure. The curve parameterized by \eqref{1007} with $ k>0 $ and $ k<0 $ correspond  to the left-handed helical structure and right-handed helical structure, respectively.  Theorem \ref{thm01}  shows the desingularization of a left-handed helix. For the case $ k<0 $, results are similar.
\end{remark}

One can also construct multiple traveling-rotating helical vortices  in $ B_{R^*}(0)\times \mathbb{R} $ with polygonal symmetry. Let us consider the curve $ \gamma(\tau) $ parameterized by \eqref{1007}. For any integer $ m $, define for
$ i=1\cdots,m $ the curves $ \gamma_i(\tau) $ parameterized by
\begin{equation}\label{1008}
\gamma_i(s,\tau) = \bar{Q}_{\frac{2\pi(i-1)}{m}}\gamma(s,\tau).
\end{equation}
The following result generalizes that of Theorem \ref{thm01} to helical vortices concentrating near multiple helices with polygonal symmetry.
\begin{theorem}\label{thm02}
Let $ k>0 $, $ c>0 $ and $ r_*\in (0,R^*) $ be any given numbers and $ m\geq 2 $ be an integer. Let  $ \gamma_i(\tau) $ be the helix parameterized by   \eqref{1008}. Then for any $ \varepsilon\in (0,\varepsilon_0] $ for some $ \varepsilon_0>0 $, there exists a classical solution pair $ (\mathbf{v}_\varepsilon, P_\varepsilon)(x,t)\in C^1(B_{R^*}(0)\times \mathbb{R}\times \mathbb{R}^+) $ of \eqref{Euler eq} such that the support set of  $ \mathbf{w}_\varepsilon $ is a collection of $ m $ topological traveling-rotating helical tubes that does not change form and  for all $ \tau $,
\begin{equation*}
\mathbf{w}_\varepsilon(\cdot,|\ln\varepsilon|^{-1}\tau)\to c\sum_{i=1}^m \delta_{\gamma_i(\tau)}\mathbf{t}_{\gamma_i(\tau)},\ \ \text{as}\ \varepsilon\to0.
\end{equation*}
Moreover, one has
\begin{enumerate}
	\item $ \mathbf{v}_\varepsilon\cdot \mathbf{n}=0$ on $\partial B_{R^*}(0)\times \mathbb{R}. $
	\item Define $ A_{i, \varepsilon}=supp(\mathbf{w}_\varepsilon)\cap B_{\bar{\rho}}\left( \bar{Q}_{\frac{2\pi(i-1)}{m}}\left(r_*,0\right) \right) \times\{0\} $ for some  small constant $ \bar{\rho}>0 $. Then there are $ R_1,R_2>0 $ such that
	\begin{equation*}
	R_1\varepsilon\leq diam(A_{i, \varepsilon})\leq R_2\varepsilon.
	\end{equation*}
\end{enumerate}
\end{theorem}

The paper is organized as follows. In section 2, we deduce the 2D vorticity-stream equations of left-handed helical solutions of \eqref{Euler eq2} and the associated semilinear elliptic equations. A generalized toy model (see \eqref{eq1-1}) and the corresponding desingularization result (see Theorem \ref{thm1}) are introduced.  In section 3, we show the approximate solutions and some basic estimates. In section 4
and section 5, we give proof of Theorem \ref{thm1}. The outline of proofs for  Theorem \ref{thm01} and Theorem \ref{thm02} are given in section 6.

\section{Solution with helical symmetry and a generalized model}
Let us first define left-handed helical symmetric solutions and reduce  \eqref{Euler eq2} to a 2D vorticity-stream model,  see  \cite{DDMW,Du,ET}. Let $ k>0 $. Define a one-parameter group $ \mathcal{G}_k=\{H_{\bar{\rho}}:\mathbb{R}^3\to\mathbb{R}^3 \} $, where
\begin{equation*}
H_{\rho}(x_1,x_2,x_3)^t=(x_1\cos\rho+x_2\sin\rho, -x_1\sin\bar{\rho+x_2\cos\rho}, x_3+k\rho)^t.
\end{equation*}
So $ H_{\rho} $ is a superposition of a rotation in $x_1Ox_2$ plane and a translation in $x_3$ axis, that is,  $ H_{\rho}(x)=\bar{Q}_{\rho}(x)+k\rho(0,0,1) $.   Clearly, $ B_{R^*}(0)\times \mathbb{R} $ is invariant under the group $ \mathcal{G}_k $.

Define a vector field
\begin{equation*}
\overrightarrow{\zeta}= (x_2, -x_1, k)^t.
\end{equation*}
Then $ \overrightarrow{\zeta}  $ is the field of tangents of symmetry lines of $ \mathcal{G}_k $.

Let us define helical functions and vector fields. A scalar function $ h $ is called a $ helical $ function, if $ h(H_{\rho}(x))=h(x) $ for any $ \rho\in\mathbb{R}, x\in B_{R^*}(0)\times \mathbb{R}. $ By direct computations it is easy to see that a $ C^1 $ function $ h $ is helical if and only if
\begin{equation*}
\overrightarrow{\zeta}\cdot\nabla h=0.
\end{equation*}

A vector field $ \mathbf{h}=(h_1,h_2,h_3) $ is called a $ helical $ field, if $ \mathbf{h}(H_{\rho}(x))=R_\rho \mathbf{h}(x) $ for any $ \rho\in\mathbb{R}, x\in B_{R^*}(0)\times \mathbb{R}. $ Direct computation shows that a $ C^1 $ vector field $ \mathbf{h} $ is helical if and only if
\begin{equation*}
\overrightarrow{\zeta}\cdot\nabla \mathbf{h}=\mathcal{R}\mathbf{h},
\end{equation*}
where $ \mathcal{R}=\begin{pmatrix}
0 & 1 & 0 \\
-1 &0 & 0 \\
0 & 0 & 0
\end{pmatrix} $ (see \cite{ET}). Helical solutions of \eqref{Euler eq} are then defined as follows.
\begin{definition}
	A function pair ($\mathbf{v}, P$) is called a $ helical$ solution  of \eqref{Euler eq} in $ B_{R^*}(0)\times \mathbb{R} $, if ($\mathbf{v}, P$) satisfies \eqref{Euler eq} and both vector field $ \mathbf{v} $ and scalar function $ P $ are helical.
\end{definition}
Throughout this paper, helical solutions also need to satisfy  the $ orthogonality~condition $:
\begin{equation}\label{ortho}
\mathbf{v}\cdot \overrightarrow{\zeta}=0,
\end{equation}
that is, the velocity field and $ \overrightarrow{\zeta} $ are orthogonal.

Under the condition \eqref{ortho}, one can check that  the vorticity field $ \mathbf{w} $ satisfies (see \cite{ET})
\begin{equation}\label{w formula}
\mathbf{w}=\frac{w}{k}\overrightarrow{\zeta},
\end{equation}
where $ w=w_3=\partial_{x_1}v_2-\partial_{x_2}v_1 $, the third component of vorticity field $ \mathbf{w} $, is a helical function. Moreover, the first equation of the vorticity equations \eqref{Euler eq2} is equivalent to
\begin{equation*}
\partial_t \mathbf{w}+(\mathbf{v}\cdot \nabla)\mathbf{w}+\frac{1}{k}w \mathcal{R}\mathbf{v}=0.
\end{equation*}
As a consequence, $ w $ satisfies
\begin{equation}\label{vor eq}
\partial_t w+(\mathbf{v}\cdot \nabla)w=0.
\end{equation}

We now introduce a $ stream~function $ and reduce the system \eqref{Euler eq2} to a 2D vorticity-stream equation. Since $ \mathbf{v} $ is a helical vector field, we have $ \overrightarrow{\zeta}\cdot\nabla \mathbf{v}=\mathcal{R}\mathbf{v} $, which implies that \begin{equation}\label{101}
x_2\partial_{x_1}v_3-x_1\partial_{x_2}v_3+k\partial_{x_3}v_3=0.
\end{equation}
The orthogonal condition shows that
\begin{equation}\label{102}
x_2v_1-x_1v_2+kv_3=0.
\end{equation}
It follows from the incompressible condition, \eqref{101} and \eqref{102} that
\begin{equation*}
\begin{split}
0=&\partial_{x_1}v_1+\partial_{x_2}v_2+\partial_{x_3}v_3=\partial_{x_1}v_1+\partial_{x_2}v_2-\frac{x_2}{k}\partial_{x_1}v_3+\frac{x_1}{k}\partial_{x_2}v_3\\
=&\partial_{x_1}v_1+\partial_{x_2}v_2-\frac{x_2}{k^2}\partial_{x_1}(-x_2v_1+x_1v_2)+\frac{x_1}{k^2}\partial_{x_2}(-x_2v_1+x_1v_2)\\
=&\frac{1}{k^2}\partial_{x_1}[(k^2+x_2^2)v_1-x_1x_2v_2]+\frac{1}{k^2}\partial_{x_2}[(k^2+x_1^2)v_2-x_1x_2v_1].
\end{split}
\end{equation*}
Since $ B_{R^*}(0) $ is simply-connected, we can define a stream function $ \varphi:B_{R^*}(0)\to \mathbb{R} $ such that $ \partial_{x_2}\varphi=\frac{1}{k^2}[(k^2+x_2^2)v_1-x_1x_2v_2],   \partial_{x_1}\varphi=-\frac{1}{k^2}[(k^2+x_1^2)v_2-x_1x_2v_1]$, that is,
\begin{equation*}
\begin{pmatrix}
\partial_{x_1}\varphi  \\
\partial_{x_2}\varphi
\end{pmatrix}=-\frac{1}{k^2}
\begin{pmatrix}
-x_1x_2 & k^2+x_1^2 \\
-(k^2+x_2^2) & x_1x_2
\end{pmatrix}
\begin{pmatrix}
v_1   \\
v_2
\end{pmatrix},
\end{equation*}
or equivalently,
\begin{equation}\label{103}
\begin{pmatrix}
v_1   \\
v_2
\end{pmatrix}
=-\frac{1}{k^2+x_1^2+x_2^2}
\begin{pmatrix}
x_1x_2 & -k^2-x_1^2 \\
k^2+x_2^2 & -x_1x_2
\end{pmatrix}
\begin{pmatrix}
\partial_{x_1}\varphi  \\
\partial_{x_2}\varphi
\end{pmatrix}.
\end{equation}

By the definition of $ w $ and \eqref{103}, we get
\begin{equation}\label{voreq1}
\begin{split}
w=&\partial_{x_1}v_2-\partial_{x_2}v_1=(-\partial_{x_2}, \partial_{x_1})\begin{pmatrix}
v_1   \\
v_2
\end{pmatrix}\\
=&(-\partial_{x_2}, \partial_{x_1})\left( -\frac{1}{k^2+x_1^2+x_2^2}
\begin{pmatrix}
x_1x_2 & -k^2-x_1^2 \\
k^2+x_2^2 & -x_1x_2
\end{pmatrix}
\begin{pmatrix}
\partial_{x_1}\varphi  \\
\partial_{x_2}\varphi
\end{pmatrix}\right)\\
=& -(\partial_{x_1}, \partial_{x_2})\left( \frac{1}{k^2+x_1^2+x_2^2}
\begin{pmatrix}
k^2+x_2^2 & -x_1x_2 \\
-x_1x_2 &  k^2+x_1^2
\end{pmatrix}
\begin{pmatrix}
\partial_{x_1}\varphi  \\
\partial_{x_2}\varphi
\end{pmatrix}\right)\\
=&\mathcal{L}_H\varphi,
\end{split}
\end{equation}
where $ \mathcal{L}_H\varphi=-\text{div}(K_H(x_1,x_2)\nabla\varphi) $ is a second order elliptic operator of divergence type with the coefficient matrix
\begin{equation}\label{coef matrix}
K_H(x_1,x_2)=\frac{1}{k^2+x_1^2+x_2^2}
\begin{pmatrix}
k^2+x_2^2 & -x_1x_2 \\
-x_1x_2 &  k^2+x_1^2
\end{pmatrix}.
\end{equation}
Clearly from the definition of the matrix $ K_H $, one has
\begin{enumerate}
	\item[(1).] $ K_H $ is a positive definite matrix and $(K_{H}(x))_{ij}\in  C^{\infty}(\overline{B_{R^*}(0)}) $ for $  i,j=1, 2. $
	\item[(2).] $ \mathcal{L}_H $ is uniformly elliptic, namely, $ \lambda_1=1, \lambda_2=\frac{k^2}{k^2+|x|^2} $ are two eigenvalues of $ K_H $ which have positive lower and upper bounds.
\end{enumerate}

From \eqref{vor eq}, \eqref{102} and \eqref{103}, one has
\begin{equation}\label{voreq2}
\begin{split}
0=&\partial_t w+v_1\partial_{x_1}w+v_2\partial_{x_2}w+v_3\partial_{x_3}w\\
=&\partial_t w+v_1\partial_{x_1}w+v_2\partial_{x_2}w+\frac{1}{k}(-x_2v_1+x_1v_2)\cdot\frac{1}{k}(-x_2\partial_{x_1}w+x_1\partial_{x_2}w)\\
=&\partial_t w+\frac{1}{k^2}(v_1,v_2)\begin{pmatrix}
k^2+x_2^2 & -x_1x_2 \\
-x_1x_2 &  k^2+x_1^2
\end{pmatrix}
\begin{pmatrix}
\partial_{x_1}w  \\
\partial_{x_2}w
\end{pmatrix}\\
=& \partial_t w-\frac{1}{k^2(k^2+x_1^2+x_2^2)}
(\partial_{x_1}\varphi, \partial_{x_2}\varphi)
\begin{pmatrix}
x_1x_2 & k^2+x_2^2 \\
-k^2-x_1^2 & -x_1x_2
\end{pmatrix}
\begin{pmatrix}
k^2+x_2^2 & -x_1x_2 \\
-x_1x_2 &  k^2+x_1^2
\end{pmatrix}
\begin{pmatrix}
\partial_{x_1}w  \\
\partial_{x_2}w
\end{pmatrix}\\
=&\partial_t w-(\partial_{x_1}\varphi, \partial_{x_2}\varphi)\begin{pmatrix}
0 & 1 \\
-1 &  0
\end{pmatrix}
\begin{pmatrix}
\partial_{x_1}w  \\
\partial_{x_2}w
\end{pmatrix}\\
=&\partial_t w+\partial_{x_2}\varphi\partial_{x_1}w-\partial_{x_1}\varphi\partial_{x_2}w\\
=&\partial_t w+\nabla^\perp\varphi\cdot \nabla w,
\end{split}
\end{equation}
where $ \perp $ denotes the clockwise rotation through $ \pi/2 $, i.e., $ (a,b)^\perp=(b,-a) $. As for the boundary condition of $ \varphi $,   it follows from $ \mathbf{v}\cdot \mathbf{n}=0 $ on $ \partial B_{R^*}(0)\times \mathbb{R} $ that (see (2.66), \cite{ET}) $ \varphi $ is a constant on $ \partial B_{R^*}(0). $ Without loss of generality, we set $ \varphi|_{\partial B_{R^*}(0)}=0. $ Thus  the 2D vorticity-stream equations of \eqref{Euler eq2} in $ B_{R^*}(0)\times \mathbb{R} $ is
\begin{equation}\label{vor str eq}
\begin{cases}
\partial_t w+\nabla^\perp\varphi\cdot \nabla w=0,\ \ &\text{in}\ B_{R^*}(0),\\
w=\mathcal{L}_H\varphi,\ \ &\text{in}\ B_{R^*}(0),\\
\varphi=0,\ \ &\text{on}\ \partial B_{R^*}(0).
\end{cases}
\end{equation}
For a solution pair $ (w,\varphi) $ of \eqref{vor str eq}, one can recover left-handed  helical velocity field $ \mathbf{v} $ and vorticity field by $ \mathbf{w} $ of \eqref{Euler eq2} by using \eqref{103}, \eqref{102}, $ \mathbf{v}(x, t) = \bar{Q}_{\frac{x_3}{k}} \mathbf{v}\left( H_{-\frac{x_3}{k}}(x), t\right)  $ and \eqref{w formula}.

Let $ \alpha $ be a constant. To construct traveling-rotating helical solutions of \eqref{Euler eq2}, we look for solutions of \eqref{vor str eq} being of the form
\begin{equation}\label{104}
w(x',t)=W(\bar{R}_{-\alpha|\ln\varepsilon| t}(x')),\ \ \varphi(x',t)=\varPhi(\bar{R}_{-\alpha|\ln\varepsilon| t}(x')),
\end{equation}
where $ x'=(x_1,x_2)\in B_{R^*}(0)$. Then one computes directly  that $( W, \varPhi) $ satisfies
\begin{equation}\label{rot eq}
\begin{cases}
\nabla W\cdot \nabla^\perp \left( \varPhi-\frac{\alpha}{2}|x'|^2|\ln\varepsilon|\right) =0,\\
W=\mathcal{L}_H\varPhi,\\
\varPhi|_{\partial B_{R^*}(0)}=0.
\end{cases}
\end{equation}
So formally if
\begin{equation}\label{105}
\mathcal{L}_H\varPhi=W=f_\varepsilon\left( \varPhi-\frac{\alpha}{2}|x'|^2|\ln\varepsilon|\right) \ \ \text{in}\ B_{R^*}(0),
\end{equation}
for some function $ f_\varepsilon $, then \eqref{rot eq} automatically holds. In the sequel we write $ x' $ as $ x=(x_1,x_2) $ and look for  solutions of a semilinear elliptic equations
\begin{equation}\label{rot eq2}
\begin{cases}
-\text{div}\cdot(K_H(x)\nabla\varPhi)=\frac{1}{\varepsilon^2}\left( \varPhi-\left( \frac{\alpha}{2}|x|^2+\beta\right) |\ln\varepsilon|\right)^p_+,\  &x\in B_{R^*}(0),\\
\varPhi(x)=0,\ &x\in \partial B_{R^*}(0),
\end{cases}
\end{equation}
where $ p>1 $, $ \alpha,\beta $ are constants to be determined later. For a solution $ \varPhi $ of \eqref{rot eq2}, one can get a rotating-invariant solution pair $(w, \varphi)$ of \eqref{vor str eq} with angular velocity $ \alpha|\ln\varepsilon| $ by simply using \eqref{105} and \eqref{104}.

Inspired by equations \eqref{rot eq2}, let us study the existence of solutions concentrating around a couple of points to a more general model 
\begin{equation}\label{eq1-1}
\begin{cases}
-\varepsilon^2\text{div}(K(x)\nabla u)= (u-q|\ln\varepsilon|)^{p}_+,\ \ &x\in \Omega,\\
u=0,\ \ &x\in\partial \Omega,
\end{cases}
\end{equation}
where $ \Omega\subset \mathbb{R}^2 $ is a simply-connected bounded domain with smooth boundary, $ \varepsilon\in(0,1) $ and $ p>1 $. $ K=(K_{i,j})_{2\times2} $ is a positive definite matrix satisfying
\begin{enumerate}
	\item[($\mathcal{K}$1).]  $ K_{i,j}(x)\in C^{\infty}(\overline{\Omega}) $ for $ 1\le i,j\le 2. $
	\item[($\mathcal{K}$2).] $ -\text{div}(K(x)\nabla \cdot) $ is a uniformly elliptic operator, that is, there exist  $ \Lambda_1,\Lambda_2>0 $ such that $$ \Lambda_1|\zeta|^2\le (K(x)\zeta|\zeta) \le \Lambda_2|\zeta|^2,\ \ \ \ \forall\ x\in \Omega, \ \zeta\in \mathbb{R}^2.$$
\end{enumerate}
 $ q(x) $ is a function defined in $ \overline{\Omega} $ satisfying
\begin{enumerate}
	\item[(Q1).]  $ q(x)\in C^{\infty}(\overline{\Omega}) $ and $ q(x)>0 $ for any $ x\in\overline{\Omega}. $
\end{enumerate}
Denote  $ det(K) $ the determinant of $ K $.
\begin{theorem}\label{thm1}
Let $ K  $ satisfy $(\mathcal{K}1) $-$(\mathcal{K}2)$ and $ q $ satisfy $ (Q1)$. Then, for any given $ m $ distinct strict local minimum (maximum) points $ x_{0,j}(j=1,\cdots, m)$ of $ q^2\sqrt{det(K)} $ in $ \Omega $, there exists $ \varepsilon_0>0 $, such that for every $ \varepsilon\in(0,\varepsilon_0] $, \eqref{eq1-1} has a solution $ u_\varepsilon $. Moreover, the following properties hold
\begin{enumerate}
	\item  Define the  set $ \bar{A}_{\varepsilon,i}=\left\{u_\varepsilon>q\ln\frac{1}{\varepsilon}\right\}\cap B_{\bar{\rho}}(x_{0,i})  $, where $ \bar{\rho} $ is small. Then there exist $ (z_{1,\varepsilon}, \cdots, z_{m,\varepsilon}) $ and $ R_1,R_2>0 $ independent of $ \varepsilon $ satisfying
\begin{equation*}
\lim_{\varepsilon\to 0}(z_{1,\varepsilon}, \cdots, z_{m,\varepsilon})=(x_{0,1}, \cdots, x_{0,m}),
\end{equation*}	
\begin{equation*}
B_{R_1\varepsilon}(z_{i,\varepsilon})\subseteq \bar{A}_{\varepsilon,i}\subseteq B_{R_2\varepsilon}(z_{i,\varepsilon}).
\end{equation*}
\item Define $ \kappa_i(u_\varepsilon)=\frac{1}{\varepsilon^2}\int_{B_{\bar{\rho}}(x_{0,i})}\left( u_\varepsilon-q\ln\frac{1}{\varepsilon}\right)^{p}_+dx. $  Then
\begin{equation*}
\lim_{\varepsilon\to 0}\kappa_i(u_\varepsilon)=2\pi q\sqrt{det(K)}(x_{0,i}).
\end{equation*}
\end{enumerate}
\end{theorem}

To obtain good estimates for approximate solutions, we change \eqref{eq1-1} to the following equivalent problem. Set $ \delta=\varepsilon|\ln\varepsilon|^{-\frac{p-1}{2}} $ and $ u=|\ln\varepsilon|w $, then \eqref{eq1-1} becomes
\begin{equation}\label{eq1}
\begin{cases}
-\delta^2\text{div}(K(x)\nabla w)=  (w-q)^{p}_+,\ \ &x\in \Omega,\\
w=0,\ \ &x\in\partial \Omega.
\end{cases}
\end{equation}
We will construct multi-peak solutions of \eqref{eq1} in sections 3-5.

\begin{remark}
Results of Theorem \ref{thm1} can be regarded as a generalization of the desingularization of classical planar vortex case (see \cite{LYY,SV}) and the vortex ring case (see  \cite{DV}). Note that the cases of planar vortices and  vortex rings correspond to the coefficient matrix $ K_H(x)=Id $ and $ \frac{1}{x_1}Id $, respectively.  In \cite{DV}, by considering   solutions of
\begin{equation*}
\begin{cases}
-\text{div}\left( \frac{1}{b}\nabla u\right) =\frac{1}{\varepsilon^2} b\left(u-q\ln\frac{1}{\varepsilon}\right)^{p-1}_+,\ \ &x\in \Omega,\\	u=0,\ \ &x\in\partial \Omega,
\end{cases}
\end{equation*}
where $ b $ is a scalar function, the authors constructed a family of   $ C^1 $ solutions $ u_\varepsilon $  with nonvanishing circulation  concentrating near a minimizer of $ q^2/b $ as $ \varepsilon\to 0 $. Indeed, if we choose $ K_H(x)=\frac{1}{b}Id $,   then by Theorem \ref{thm1} solutions will shrink to minimizers of $ q^2\sqrt{det(K_H)}=q^2/b $, which  coincides with the results in \cite{DV}.
\end{remark}

\begin{remark}
Recently, \cite{CW} considered desingularization of steady solutions to 3D   Euler equation \eqref{Euler eq2} with helical symmetry in helical domains. Using the critical point theory and the estimates of capacity, \cite{CW} proved the existence and asymptotic behavior of ground state solutions of \eqref{eq1-1} concentrating near a single point. While in this paper,  by using finite-dimensional reduction method, we construct multi-peak solutions concentrating near a collection of  given points, which extends the results in \cite{CW}.
\end{remark}

\section{Approximate solutions}

Our aim is to solve the following equations
\begin{equation*}
\begin{cases}
-\delta^2\text{div}(K(x)\nabla w)=  (w-q)^{p}_+,\ \ &\text{in}\  \Omega,\\
w=0,\ \ &\text{on}\ \partial \Omega.
\end{cases}
\end{equation*}
Note that $ -\text{div}(K(x)\nabla \cdot) $ is a uniformly elliptic operator. Throughout this paper, we denote $ C,C_1,C_2\cdots $ positive constants independent of $ \varepsilon $, whose values may change from line to line.

First, since $ K $ is a $ C^\infty $ positive definite matrix with all eigenvalues having uniformly positive lower and upper bounds, by the Cholesky decomposition one can find a matrix-valued function $ T\in C^\infty(\overline{\Omega}) $ such that for any $ x\in \Omega $, $ T(x) $ is invertible and
\begin{equation}\label{T_z choice}
(T(x)^{-1})(T(x)^{-1})^t=K(x).
\end{equation}
For simplicity, we denote $ T_{x}=T(x) $.

Let $ R>1 $ be a large constant satisfying $ \Omega\subseteq T_{x}^{-1}(B_R(0))+x  $ for any $ x\in \Omega. $ Clearly by  the positive definiteness of  $  K $, such $ R $ exists.

Consider
\begin{equation}\label{eq2}
\begin{cases}
-\delta^2\Delta w=(w-a)^{p}_+,\ \ &\text{in}\ B_R(0),\\
w=0,\ \  &\text{on}\ \partial B_R(0),
\end{cases}
\end{equation}
where $ a>0 $ is a constant. One computes directly that the unique $ C^1 $ positive solution of \eqref{eq2} is
\begin{equation*}
W_{\delta,a}(x)=\begin{cases}
a+\delta^{\frac{2}{p-1}}s_\delta^{-\frac{2}{p-1}}\phi\left(\frac{|x|}{s_\delta}\right),\ \ &|x|\le s_\delta,\\
a\ln\frac{|x|}{R}/\ln\frac{s_\delta}{R},\ \  &s_\delta\leq |x|\le R,
\end{cases}
\end{equation*}
where $ \phi\in H^1_0(B_1(0)) $ satisfies
\begin{equation*}
-\Delta\phi=\phi^p, \ \ \phi>0\ \ \text{in}\ B_1(0),
\end{equation*}
and $ s_\delta $ satisfy the relation
\begin{equation}\label{201}
\delta^{\frac{2}{p-1}}s_\delta^{-\frac{2}{p-1}}\phi'(1)=a/\ln\frac{s_\delta}{R}.
\end{equation}
Hence \eqref{201} is uniquely solvable if $ \delta>0 $ is sufficiently small and
\begin{equation*}
\frac{s_\delta}{\delta|\ln\delta|^{\frac{p-1}{2}}}\to \left( \frac{|\phi'(1)|}{a}\right) ^{\frac{p-1}{2}}\ \ \ \ \text{as}\ \delta\to0.
\end{equation*}
The Pohazaev identity implies
\begin{equation}\label{PI}
\int_{B_1(0)}\phi^{p+1}=\frac{\pi(p+1)}{2}|\phi'(1)|^2,\ \ \int_{B_1(0)}\phi^{p}= 2\pi|\phi'(1)|.
\end{equation}
Now for any $ \hat{x}\in \Omega, \hat{q}>0 $, let $ V_{\delta, \hat{x}, \hat{q}} $ be a $ C^1 $ positive solution of  the following equations
\begin{equation}\label{eq3}
\begin{cases}
-\delta^2\text{div}(K(\hat{x})\nabla v)=  (v-\hat{q})^{p}_+,\ \ & \text{in}\  T_{\hat{x}}^{-1}(B_R(0)),\\
v=0,\ \ &\text{on} \ \partial T_{\hat{x}}^{-1}(B_R(0)).
\end{cases}
\end{equation}
Thus one has $ V_{\delta, \hat{x}, \hat{q}}(x)=W_{\delta, \hat{q}}(T_{\hat{x}}x) $. Indeed, let $ u(x)=v(T_{\hat{x}}^{-1}x)\in H^1_0(B_R(0)) $. Then $ u $ satisfies \eqref{eq2} with $ a=\hat{q}. $ So $ u=W_{\delta,\hat{q}} $, which implies that $ v(x)=W_{\delta, \hat{q}}(T_{\hat{x}}x) $. Clearly $ V_{\delta, \hat{x}, \hat{q}} $ has an explicit profile
\begin{equation*}
V_{\delta, \hat{x}, \hat{q}}(x)=\begin{cases}
\hat{q}+\delta^{\frac{2}{p-1}}s_\delta^{-\frac{2}{p-1}}\phi\left(\frac{|T_{\hat{x}}x|}{s_\delta}\right),\ \ &|T_{\hat{x}}x|\le s_\delta,\\
\hat{q}\ln\frac{|T_{\hat{x}}x|}{R}/\ln\frac{s_\delta}{R},\ \  &s_\delta\leq |T_{\hat{x}}x|\le R.
\end{cases}
\end{equation*}

For any $ z\in \Omega,  $ define
\begin{equation*}
V_{\delta, \hat{x}, \hat{q}, z}(x):=V_{\delta, \hat{x}, \hat{q}}(x-z),\ \ \forall x\in\Omega.
\end{equation*}

Since $ V_{\delta, \hat{x}, \hat{q}, z} $ is not 0 on $ \partial \Omega $, we need to make a projection of $ V_{\delta, \hat{x}, \hat{q}, z} $ on $ H^1_0(\Omega) $. Let $ PV_{\delta, \hat{x}, \hat{q}, z} $ be a solution of
\begin{equation}\label{eq4}
\begin{cases}
-\delta^2\text{div}(K(\hat{x})\nabla v)=  (V_{\delta, \hat{x}, \hat{q}, z}-\hat{q})^{p}_+,\ \ & \text{in}\  \Omega,\\
v=0,\ \ &\text{on}\ \partial \Omega.
\end{cases}
\end{equation}

We claim that for $ \delta $ sufficiently small,
\begin{equation}\label{202}
PV_{\delta, \hat{x}, \hat{q}, z}(x)=V_{\delta, \hat{x}, \hat{q}, z}(x)-\frac{\hat{q}}{\ln\frac{R}{s_\delta}}g_{\hat{x}}(T_{\hat{x}}x, T_{\hat{x}}z),\ \ \forall x\in\Omega,
\end{equation}
where $ g_{\hat{x}}(x, y)=2\pi h_{\hat{x}}(x, y)+\ln R$ for any $ x,y\in T_{\hat{x}}(\Omega) $, and $ h_{\hat{x}}(x, y) $ is the regular part of Green's function of $ -\Delta $ on $ T_{\hat{x}}(\Omega) $, namely for any $ y\in T_{\hat{x}}(\Omega) $,
\begin{equation}\label{h profile}
\begin{cases}
-\Delta h_{\hat{x}}(x, y)=0, \ \ &x\in T_{\hat{x}}(\Omega),\\
h_{\hat{x}}(x, y)=\frac{1}{2\pi}\ln\frac{1}{|x-y|}, \ \ &x\in \partial T_{\hat{x}}(\Omega).
\end{cases}
\end{equation}
Note that  the Green's function $ G_{\hat{x}}(x, y) $ of $ -\Delta $ in $ T_{\hat{x}}(\Omega) $ with  Dirichlet zero boundary condition    has the decomposition
\begin{equation}\label{G profile}
G_{\hat{x}}(x, y)=\frac{1}{2\pi}\ln\frac{1}{|x-y|}-h_{\hat{x}}(x, y),\ \ \forall x,y\in T_{\hat{x}}(\Omega).
\end{equation}
Indeed, by \eqref{eq3} and \eqref{eq4} one has
\begin{equation*}
\begin{cases}
-\delta^2\text{div}(K(\hat{x})\nabla (V_{\delta, \hat{x}, \hat{q}, z}-PV_{\delta, \hat{x}, \hat{q}, z}))(x)= 0,\ \ &  x\in\Omega,\\
V_{\delta, \hat{x}, \hat{q}, z}-PV_{\delta, \hat{x}, \hat{q}, z}=\hat{q}\ln\frac{|T_{\hat{x}}(x-z)|}{R}/\ln\frac{s_\delta}{R},\ \ &x\in\partial \Omega.
\end{cases}
\end{equation*}
Define $ \check{u}(y)=(V_{\delta, \hat{x}, \hat{q}, z}-PV_{\delta, \hat{x}, \hat{q}, z})(T_{\hat{x}}^{-1}y), y\in T_{\hat{x}}(\Omega)  $. Then
\begin{equation*}
\begin{cases}
-\delta^2\Delta \check{u}(y)= 0,\ \ &  y\in T_{\hat{x}}(\Omega),\\
\check{u}(y)=\hat{q}\ln\frac{|y-T_{\hat{x}}z|}{R}/\ln\frac{s_\delta}{R},\ \ &y\in\partial T_{\hat{x}}(\Omega).
\end{cases}
\end{equation*}
So $ \check{u}(y)=\frac{\hat{q}}{\ln\frac{s_\delta}{R}}(-2\pi h_{\hat{x}}(y, T_{\hat{x}}z)-\ln R) $, which implies that for any $ x\in \Omega $
\begin{equation*}
(V_{\delta, \hat{x}, \hat{q}, z}-PV_{\delta, \hat{x}, \hat{q}, z})(x)=\check{u}(T_{\hat{x}}x)=\frac{\hat{q}}{\ln\frac{s_\delta}{R}}(-2\pi h_{\hat{x}}(T_{\hat{x}}x, T_{\hat{x}}z)-\ln R)=\frac{\hat{q}}{\ln\frac{R}{s_\delta}}g_{\hat{x}}(T_{\hat{x}}x, T_{\hat{x}}z).
\end{equation*}
We get \eqref{202}.

In the following, we will construct solutions of the form
\begin{equation*}
\sum_{j=1}^mPV_{\delta, \hat{x}_j, \hat{q}_j, z_j}+\omega_\delta,
\end{equation*}
where $ \Sigma_{j=1}^mPV_{\delta, \hat{x}_j, \hat{q}_j, z_j} $ is the main term and $ \omega_\delta $ is an error term. To make the  norm of $ \omega_\delta $ as small as possible, we need to choose $ \hat{x}_j$  and $ \hat{q}_j $ suitably close to $ z_j $ and $ q(z_j) $.

Let $ (x_{0,1},\cdots,x_{0,m}) $ be $ m $ distinct strict local maximum points (or minimum points) of $ q^2\sqrt{det(K)} $ in $ \Omega $. Hence we can choose $ \bar{\rho}>0 $ sufficiently small such that
\begin{equation*}
B_{\bar{\rho}}(x_{0,i})\Subset\Omega, \ \overline{B_{\bar{\rho}}(x_{0,i})}\cap\overline{B_{\bar{\rho}}(x_{0,i})}=\varnothing,\  \ \forall 1\leq i\neq j\leq m.
\end{equation*}

Define the admissible set $   \mathcal{M}\subseteq\mathbb{R}^{(2m)}$ satisfying
\begin{equation}\label{admis set}
\mathcal{M}=\{Z=(z_1,z_2,\cdots, z_m)\in\mathbb{R}^{(2m)}\mid z_i\in B_{\bar{\rho}}(x_{0,i}), \ i=1,\cdots,m\}.
\end{equation}

Let
$$ \hat{x}_i=z_i$$
and  $ \hat{q}_i=\hat{q}_{\delta,i}(Z) $, $ i=1,\cdots,m $, be the solution the equations
\begin{equation}\label{q_i choice}
\hat{q}_i=q(z_i)+\frac{\hat{q}_i}{\ln\frac{R}{\varepsilon}}g_{z_i}(T_{z_i}z_i, T_{z_i}z_i)-\Sigma_{j\neq i}\frac{\hat{q}_j}{\ln\frac{R}{\varepsilon}}\bar{G}_{z_j}(T_{z_j}z_i, T_{z_j}z_j),
\end{equation}
where $ \bar{G}_{z_j}(x,y)=\ln\frac{ R}{|x-y|}-g_{z_j}(x,y)=2\pi G_{z_j}(x,y) $ for any $ x,y\in T_{z_j}(\Omega) $.

It follows from Lemma \ref{lemA-3} in Appendix  that for any $ Z $ satisfying \eqref{admis set} and $ \delta $ sufficiently small,  there exist $ \hat{q}_{\delta,i}(Z) $ satisfying \eqref{q_i choice}. Moreover, one has
\begin{equation*}
\hat{q}_i=\frac{q(z_i)-\Sigma_{j\neq i}\frac{\hat{q}_j}{\ln\frac{R}{\varepsilon}}\bar{G}_{z_j}(T_{z_j}z_i, T_{z_j}z_j)}{1-\frac{1}{\ln\frac{R}{\varepsilon}}g_{z_i}(T_{z_i}z_i, T_{z_i}z_i)}.
\end{equation*}

For $ Z=(z_1,\cdots,z_m) $, denote
\begin{equation*}
V_{\delta,Z,j}=PV_{\delta, z_j, \hat{q}_{\delta,j}, z_j},\ \ V_{\delta,Z}=\sum_{j=1}^mV_{\delta,Z,j}.
\end{equation*}
Let $ s_{\delta,j} $ satisfy
\begin{equation*}
\delta^{\frac{2}{p-1}}s_{\delta,j}^{-\frac{2}{p-1}}\phi'(1)=\hat{q}_{\delta,j}/\ln\frac{s_{\delta,j}}{R}.
\end{equation*}
Then  one can easily verify that
\begin{equation}\label{2000}
\frac{1}{\ln\frac{R}{s_{\delta,j}}}=\frac{1}{\ln\frac{R}{\varepsilon}}+O\left( \frac{\ln|\ln\varepsilon|}{|\ln\varepsilon|^2}\right) .
\end{equation}

By the choice of $ \hat{x}_j, \hat{q}_{\delta,j} $, we claim that for any fixed constant $ L>0 $ and $ x\in B_{Ls_{\delta,i}}(z_i) $,
\begin{equation}\label{203}
\begin{split}
V_{\delta,Z}(x)-q(x)=V_{\delta, z_i, \hat{q}_{\delta,i}, z_i}(x)-\hat{q}_{\delta,i}+O\left( \frac{\ln|\ln\varepsilon|}{|\ln\varepsilon|^2}\right).
\end{split}
\end{equation}
Indeed, we have for any $ x\in B_{Ls_{\delta,i}}(z_i) $,
\begin{equation*}
\begin{split}
&V_{\delta,Z,i}(x)-q(x)\\
=&V_{\delta, z_i, \hat{q}_{\delta,i}, z_i}(x)-\frac{\hat{q}_{\delta,i}}{\ln\frac{R}{s_{\delta,i}}}g_{z_i}(T_{z_i}x, T_{z_i}z_i)-q(x)\\
=&V_{\delta, z_i, \hat{q}_{\delta,i}, z_i}(x)-q(z_i)-\frac{\hat{q}_{\delta,i}}{\ln\frac{R}{s_{\delta,i}}}g_{z_i}(T_{z_i}z_i, T_{z_i}z_i)+O(s_{\delta,i})+O\left( \frac{s_{\delta,i}|\nabla g_{z_i}(T_{z_i}z_i, T_{z_i}z_i)|}{\ln\frac{R}{s_{\delta,i}}}\right) \\
=&V_{\delta, z_i, \hat{q}_{\delta,i}, z_i}(x)-q(z_i)-\frac{\hat{q}_{\delta,i}}{\ln\frac{R}{\varepsilon}}g_{z_i}(T_{z_i}z_i, T_{z_i}z_i)+ O\left( \frac{\ln|\ln\varepsilon|}{|\ln\varepsilon|^2}\right),
\end{split}
\end{equation*}
and for any $ j\neq i $, $ x\in B_{Ls_{\delta,i}}(z_i) $,
\begin{equation*}
\begin{split}
V_{\delta,Z,j}(x)=&V_{\delta, z_j, \hat{q}_{\delta,j}, z_j}(x)-\frac{\hat{q}_{\delta,j}}{\ln\frac{R}{s_{\delta,j}}}g_{z_j}(T_{z_j}x, T_{z_j}z_j)\\
=&\frac{\hat{q}_{\delta,j}}{\ln\frac{R}{s_{\delta,j}}}\bar{G}_{z_j}(T_{z_j}x, T_{z_j}z_j)\\
=&\frac{\hat{q}_{\delta,j}}{\ln\frac{R}{s_{\delta,j}}}\bar{G}_{z_j}(T_{z_j}z_i, T_{z_j}z_j)+O\left(\frac{s_{\delta,j}|\nabla \bar{G}_{z_j}(T_{z_j}z_i, T_{z_j}z_j)|}{\ln\frac{R}{s_{\delta,j}}}\right) \\
=&\frac{\hat{q}_{\delta,j}}{\ln\frac{R}{\varepsilon}}\bar{G}_{z_j}(T_{z_j}z_i, T_{z_j}z_j)+ O\left( \frac{\ln|\ln\varepsilon|}{|\ln\varepsilon|^2}\bar{G}_{z_j}(T_{z_j}z_i, T_{z_j}z_j)\right) \\
=&\frac{\hat{q}_{\delta,j}}{\ln\frac{R}{\varepsilon}}\bar{G}_{z_j}(T_{z_j}z_i, T_{z_j}z_j)+ O\left( \frac{\ln|\ln\varepsilon|}{|\ln\varepsilon|^2}\right),
\end{split}
\end{equation*}
where we have used \eqref{2000} and Lemma \ref{lemA-2} in Appendix. Adding up the above inequalities and using \eqref{q_i choice}, we get
\begin{equation*}
\begin{split}
V_{\delta,Z}(x)-q(x)=V_{\delta, z_i, \hat{q}_{\delta,i}, z_i}(x)-\hat{q}_{\delta,i}+O\left( \frac{\ln|\ln\varepsilon|}{|\ln\varepsilon|^2}\right),\ \ \forall x\in B_{Ls_{\delta,i}}(z_i).
\end{split}
\end{equation*}

From \eqref{201}, \eqref{q_i choice} and Lemma \ref{lemA-2}, we get
\begin{equation}\label{200-1}
\frac{\partial \hat{q}_{\delta,i}}{\partial z_{i,h}}=O(1),
\end{equation}
\begin{equation}\label{200-2}
\frac{\partial s_{\delta,i}}{\partial z_{i,h}}=O(\delta|\ln\delta|^{\frac{p-1}{2}}).
\end{equation}
Using the definition of $ V_{\delta, z_i, \hat{q}_{\delta,i}, z_i} $, \eqref{200-1} and \eqref{200-2}, we obtain

\begin{equation}\label{200}
\begin{split}
\frac{\partial V_{\delta, z_i, \hat{q}_{\delta,i}, z_i}(x)}{\partial z_{i,h}}=\begin{cases}
-\frac{1}{s_{\delta,i}}(\frac{\delta}{s_{\delta,i}})^{\frac{2}{p-1}}\phi'(\frac{|T_{z_i}(x-z_i)|}{s_{\delta,i}})\frac{(T_{z_i})_h^t\cdot T_{z_i}(x-z_i)}{|T_{z_i}(x-z_i)|}+O(1),\ \ &|T_{z_i}(x-z_i)|\leq s_{\delta,i},\\
\frac{\hat{q}_{\delta,i}}{\ln\frac{R}{s_{\delta,i}}}\frac{(T_{z_i})_h^t\cdot T_{z_i}(x-z_i)}{|T_{z_i}(x-z_i)|^2}+O\left( \frac{\ln\frac{R}{|T_{z_i}(x-z_i)|}}{\ln\frac{R}{s_{\delta,i}}}\right),\ \ &|T_{z_i}(x-z_i)|> s_{\delta,i},
\end{cases}
\end{split}
\end{equation}
where $ (T_{z_i})_h^t $ is the h-th row of $ (T_{z_i})^t $.
\section{The reduction}
Now we  find solution of \eqref{eq1} being of the form $$ V_{\delta, Z}+\omega_\delta. $$
First we prove that for any $ Z $ satisfying \eqref{admis set}, there exists $ \omega_{\delta,Z} $ such that $ V_{\delta, Z}+\omega_{\delta,Z} $ solves \eqref{eq1} in a co-dimensional $ 2m $ subspace of $ H^1_0 $. In the next section we choose proper $ Z=Z(\delta) $ such that $ V_{\delta, Z}+\omega_\delta $ is a solution.

Let us consider the following equation
\begin{equation}\label{eq5}
-\Delta w=w^p_+,\ \ \text{in}\ \mathbb{R}^2.
\end{equation}
The unique $ C^1 $ solution is
\begin{equation*}
w(x)=\begin{cases}
\phi(x),\ \ &|x|\leq 1,\\
\phi'(1)\ln|x|,\ \ &|x|> 1.
\end{cases}
\end{equation*}
By the classical elliptic equation theory, $ w\in C^{2,\alpha}(\mathbb{R}^2) $ for any $ \alpha\in(0,1) $. The linearized equation of \eqref{eq5} at $ w $ is
\begin{equation}\label{limit eq}
-\Delta v-pw^{p-1}_+v=0, \ \ v\in L^{\infty}(\mathbb{R}^2).
\end{equation}
Clearly, $ \frac{\partial w}{\partial x_h} $ $(h=1,2) $  are solutions of \eqref{limit eq}. It follows from \cite{DY} (see also \cite{CLW}) that
\begin{proposition}[Non-degeneracy]\label{Non-degenerate}
$ w $ is non-degenerate, i.e., the kernel of the linearized equation \eqref{limit eq} is $$ span\{\frac{\partial w}{\partial x_1}, \frac{\partial w}{\partial x_2}\}. $$
\end{proposition}

Denote
\begin{equation}\label{204}
F_{\delta,Z}=\{u\in L^p(\Omega)\mid \int_{\Omega}\frac{\partial V_{\delta, Z,j}}{\partial z_{j,h}}u=0,\ \ \forall j=1,\cdots,m,\ h=1,2\},
\end{equation}
and
\begin{equation}\label{205}
E_{\delta,Z}=\{u\in W^{2,p}\cap H^1_0(\Omega)\mid \int_{\Omega}\text{div}(K(x)\nabla \frac{\partial V_{\delta, Z,j}}{\partial z_{j,h}})u=0,\ \ \forall j=1,\cdots,m,\ h=1,2\}.
\end{equation}
So $ F_{\delta,Z}$ and $ E_{\delta,Z} $ are co-dimensional $ 2m $ subspaces of $ L^p $ and $ W^{2,p}\cap H^1_0(\Omega) $, respectively.

For any $ u\in L^p(\Omega) $, define the projection operator $ Q_\delta: L^p \to  F_{\delta,Z} $
\begin{equation}\label{206}
Q_\delta u:=u-\sum_{j=1}^m\sum_{h=1}^2C_{j,h}\frac{\partial}{\partial z_{j,h}}(-\delta^2\text{div}(K(z_j)\nabla   V_{\delta, Z,j})),
\end{equation}
where $ C_{j,h} (j=1,\cdots,m,\ h=1,2) $ satisfies
\begin{equation}\label{207}
\sum_{j=1}^m\sum_{h=1}^2C_{j,h}\int_{\Omega}\frac{\partial}{\partial z_{j,h}}(-\delta^2\text{div}(K(z_j)\nabla   V_{\delta, Z,j}))\frac{\partial V_{\delta, Z,i}}{\partial z_{i,\hbar}}=\int_{\Omega}u \frac{\partial V_{\delta, Z,i}}{\partial z_{i,\hbar}},\ \ \forall i=1,\cdots,m,\ \hbar=1,2.
\end{equation}

By Lemma \ref{lemA-4}, we know that $ Q_\delta $ is a well-defined  linear operator from $ L^p $ to $ F_{\delta,Z}$. Indeed, using \eqref{200} and Lemma \ref{lemA-2},  the coefficient matrix
\begin{equation}\label{coef of C}
\begin{split}
&\int_{\Omega}\frac{\partial}{\partial z_{j,h}}(-\delta^2\text{div}(K(z_j)\nabla   V_{\delta, Z,j}))\frac{\partial V_{\delta, Z,i}}{\partial z_{i,\hbar}}\\
=&p\int_{\Omega}(V_{\delta, z_j, \hat{q}_{\delta,j}, z_j}-\hat{q}_{\delta,j})^{p-1}_+\left( \frac{\partial V_{\delta, z_j, \hat{q}_{\delta,j}, z_j}}{\partial z_{j,h}}-\frac{\partial \hat{q}_{\delta,j}}{\partial z_{j,h}}\right) \frac{\partial V_{\delta, Z,i}}{\partial z_{i,\hbar}}\\
=&p\int_{\Omega}(V_{\delta, z_j, \hat{q}_{\delta,j}, z_j}-\hat{q}_{\delta,j})^{p-1}_+\frac{\partial V_{\delta, z_j, \hat{q}_{\delta,j}, z_j}}{\partial z_{j,h}} \frac{\partial V_{\delta, z_i, \hat{q}_{\delta,i}, z_i}}{\partial z_{i,\hbar}}+O\left( \frac{\varepsilon}{|\ln\varepsilon|^{p}}\right) \\
=&\delta_{i,j}\frac{(M_{i})_{h,\hbar}}{|\ln\varepsilon|^{p+1}}+O\left( \frac{\varepsilon}{|\ln\varepsilon|^{p}}\right) ,
\end{split}
\end{equation}
where $ \delta_{i,j}= 1 $ if $ i = j $; otherwise, $ \delta_{i,j}= 0 $.  $ M_{i} $ are $ m $ positive definite matrices and there exist positive constants $ \bar{c}_1,\bar{c}_2 $ independent of $ \delta, Z $ such that all eigenvalues of $ M_{i} $ belong to $ (\bar{c}_1,\bar{c}_2) $.  So there exists the unique $ C_{j,h} $ satisfying \eqref{207}. Note that for any $ u\in L^p $, $ Q_{\delta} u \equiv u $ in $ \Omega \backslash\cup_{i=1}^mB_{Ls_{\delta,i}}(z_{i}) $ for some $ L>1 $.

The linearized operator of \eqref{eq1} at $ V_{\delta, Z} $ is
\begin{equation*}
L_\delta \omega:=-\delta^2\text{div}(K(x)\nabla \omega)- p(V_{\delta, Z}-q)^{p-1}_+\omega.
\end{equation*}

We have the following estimates of $ L_\delta $.
\begin{lemma}\label{coercive esti}
There exist  $ \rho_0>0, \delta_1>0  $ such that for any $ \delta\in(0,\delta_1], Z$ satisfying \eqref{admis set}, $ u\in E_{\delta,Z} $ satisfying $ Q_\delta L_\delta u=0 $ in $ \Omega\backslash\cup_{j=1}^mB_{Ls_{\delta,j}}(z_j) $ for some $ L>1 $ large, then
\begin{equation*}
||Q_\delta L_\delta u||_{L^p}\geq \frac{\rho_0\varepsilon^{\frac{2}{p}}}{|\ln\varepsilon|^{p-1}}||u||_{L^\infty}.
\end{equation*}
\end{lemma}

\begin{proof}
We argue by contradiction. Suppose that there are $ \delta_N \to 0 $, $ Z_N=(z_{N,1},\cdots,z_{N,m})\to (z_{1},\cdots,z_{m}) $ satisfying (2.5) and $  u_N \in E_{\delta_N, Z_N} $ with $ Q_{\delta_N} L_{\delta_N} u_N = 0 $ in $ \Omega \backslash\cup_{j=1}^mB_{Ls_{\delta_N,j}}(z_{N,j}) $ for some $ L $ large and $ ||u_N||_{L^\infty}=1 $ such that
\begin{equation*}
||Q_{\delta_N} L_{\delta_N} u_N||_{L^p}\leq \frac{1}{N}\frac{\varepsilon_N^{\frac{2}{p}}}{|\ln\varepsilon_N|^{p-1}}.
\end{equation*}
 Let
\begin{equation}\label{208}
Q_{\delta_N} L_{\delta_N} u_N=  L_{\delta_N} u_N-\sum_{j=1}^m\sum_{h=1}^2C_{j,h,N}\frac{\partial}{\partial z_{j,h}}(-\delta_N^2\text{div}(K(z_{N,j})\nabla   V_{\delta_N, Z_N,j})).
\end{equation}
We now estimate $ C_{j,h,N} $. For fixed $ i=1,\cdots,m, \hbar=1,2 $, multiplying \eqref{208} by $ \frac{\partial V_{\delta_N, Z_N,i}}{\partial z_{i,\hbar}} $ and integrating on $ \Omega $ we get
\begin{equation*}
\begin{split}
&\int_{\Omega}u_NL_{\delta_N}\left( \frac{\partial V_{\delta_N, Z_N,i}}{\partial z_{i,\hbar}}\right) =\int_{\Omega}L_{\delta_N}u_N\frac{\partial V_{\delta_N, Z_N,i}}{\partial z_{i,\hbar}}\\
=&\sum_{j=1}^m\sum_{h=1}^2C_{j,h,N}\int_{\Omega}\frac{\partial}{\partial z_{j,h}}(-\delta_N^2\text{div}(K(z_{N,j})\nabla   V_{\delta_N, Z_N,j}))\frac{\partial V_{\delta_N, Z_N,i}}{\partial z_{i,\hbar}}.
\end{split}
\end{equation*}

We estimate $ \int_{\Omega}u_NL_{\delta_N}\left( \frac{\partial V_{\delta_N, Z_N,i}}{\partial z_{i,\hbar}}\right). $ Note that
\begin{equation}\label{301}
\begin{split}
&\int_{\Omega}u_NL_{\delta_N}\left( \frac{\partial V_{\delta_N, Z_N,i}}{\partial z_{i,\hbar}}\right) \\
=&\int_{\Omega}u_N\left[ -\delta_N^2\text{div}(K(x)\nabla \frac{\partial V_{\delta_N, Z_N,i}}{\partial z_{i,\hbar}})- p(V_{\delta_N, Z_N}-q)^{p-1}_+\frac{\partial V_{\delta_N, Z_N,i}}{\partial z_{i,\hbar}}\right]  \\
=&\int_{\Omega}u_N \frac{\partial }{\partial z_{i,\hbar}}(-\delta_N^2\text{div}(K(z_{N,i})\nabla V_{\delta_N, Z_N,i}))-\int_{\Omega}u_N \left( -\delta_N^2\text{div}\left( \frac{\partial K(z_{N,i})}{\partial z_{i,\hbar}}\nabla V_{\delta_N, Z_N,i}\right) \right) \\
&+\int_{\Omega}u_N \left( -\delta_N^2\text{div}\left( \left( K(x)-K(z_{N,i})\right) \nabla \frac{\partial V_{\delta_N, Z_N,i}}{\partial z_{i,\hbar}}\right) \right) -p\int_{\Omega}u_N(V_{\delta_N, Z_N}-q)^{p-1}_+\frac{\partial V_{\delta_N, Z_N,i}}{\partial z_{i,\hbar}}\\
=:&I_1+I_2+I_3+I_4,
\end{split}
\end{equation}
where the matrix $ \frac{\partial K(z_i)}{\partial z_{i,\hbar}} $ is defined by $ \frac{\partial K(z_i)}{\partial z_{i,\hbar}}= \begin{pmatrix}
\frac{\partial K_{1,1}(z_i)}{\partial z_{i,\hbar}} & \frac{\partial K_{1,2}(z_i)}{\partial z_{i,\hbar}} \\
\frac{\partial K_{2,1}(z_i)}{\partial z_{i,\hbar}} &  \frac{\partial K_{2,2}(z_i)}{\partial z_{i,\hbar}}
\end{pmatrix}. $

By \eqref{203}, \eqref{200-1}, Lemma \ref{lemA-2} and  Lemma \ref{lemA-5}, one has
\begin{equation}\label{302}
\begin{split}
&I_1+I_4\\
=&p\int_{\Omega}u_N\left( V_{\delta_N, z_{N,i}, \hat{q}_{\delta_N,i}, z_{N,i}}-\hat{q}_{\delta_N,i}\right)^{p-1}_+\left( \frac{\partial V_{\delta_N, z_{N,i}, \hat{q}_{\delta_N,i}, z_{N,i}}}{\partial z_{i,h}}-\frac{\partial \hat{q}_{\delta_N,i}}{\partial z_{i,h}}\right) \\
&-p\int_{\Omega}u_N\left( V_{\delta_N, z_{N,i}, \hat{q}_{\delta_N,i}, z_{N,i}}-\hat{q}_{\delta_N,i}+O\left( \frac{\ln|\ln\varepsilon_N|}{|\ln\varepsilon_N|^2}\right) \right)^{p-1}_+\frac{\partial V_{\delta_N, Z_N,i}}{\partial z_{i,\hbar}}\\
=&p\int_{\Omega}u_N\left( V_{\delta_N, z_{N,i}, \hat{q}_{\delta_N,i}, z_{N,i}}-\hat{q}_{\delta_N,i}\right)^{p-1}_+  \frac{\partial V_{\delta_N, z_{N,i}, \hat{q}_{\delta_N,i}, z_{N,i}}}{\partial z_{i,h}}\\
&-p\int_{\Omega}u_N\left( V_{\delta_N, z_{N,i}, \hat{q}_{\delta_N,i}, z_{N,i}}-\hat{q}_{\delta_N,i}+O\left( \frac{\ln|\ln\varepsilon_N|}{|\ln\varepsilon_N|^2}\right) \right)^{p-1}_+\frac{\partial V_{\delta_N, z_{N,i}, \hat{q}_{\delta_N,i}, z_{N,i}}}{\partial z_{i,h}}+O\left( \frac{\varepsilon_N^2}{|\ln\varepsilon_N|^{p-1}}\right)\\
=&O\left( \frac{\varepsilon_N\ln|\ln\varepsilon_N|}{|\ln\varepsilon_N|^{p+1}}\right).
\end{split}
\end{equation}
For $ I_2 $ and $ I_3 $, direct computation shows that
\begin{equation}\label{303}
\begin{split}
I_2=&\int_{\Omega}u_N \left( \delta_N^2\text{div}\left( \frac{\partial K(z_{N,i})}{\partial z_{i,\hbar}}\nabla V_{\delta_N, Z_N,i}\right) \right) \\
=&\int_{\Omega}u_N \left( \delta_N^2\text{div}\left( \frac{\partial K(z_{N,i})}{\partial z_{i,\hbar}}\nabla V_{\delta_N, z_{N,i}, \hat{q}_{\delta_N,i}, z_{N,i}}\right) \right) +O\left( \frac{\delta_N^2}{|\ln\varepsilon_N|}\right)\\
=&\int_{|T_{z_{N,i}}(x-z_{N,i})|\leq s_{\delta_N,i}}+\int_{|T_{z_{N,i}}(x-z_{N,i})|> s_{\delta_N,i}}u_N \left( \delta_N^2\text{div}\left( \frac{\partial K(z_{N,i})}{\partial z_{i,\hbar}}\nabla V_{\delta_N, z_{N,i}, \hat{q}_{\delta_N,i}, z_{N,i}}\right)\right) +O\left( \frac{\delta_N^2}{|\ln\varepsilon_N|}\right)\\
=&O\left( \frac{\delta_N^2}{|\ln\varepsilon_N|}\right) +O(\delta_N^2)+O\left( \frac{\delta_N^2}{|\ln\varepsilon_N|}\right)\\
=&O\left( \frac{\varepsilon_N^2}{|\ln\varepsilon_N|^{p-1}}\right),
\end{split}
\end{equation}
and
\begin{equation}\label{304}
\begin{split}
I_3=&\int_{\Omega}u_N \left( -\delta_N^2\text{div}\left( (K(x)-K(z_{N,i}))\nabla \frac{\partial V_{\delta_N, Z_N,i}}{\partial z_{i,\hbar}}\right)\right) \\
=&\int_{\Omega}u_N \left( -\delta_N^2\text{div}\left( (K(x)-K(z_{N,i}))\nabla \frac{\partial V_{\delta_N, z_{N,i}, \hat{q}_{\delta_N,i}, z_{N,i}}}{\partial z_{i,\hbar}}\right)\right) +O\left( \frac{\delta_N^2}{|\ln\varepsilon_N|}\right)\\
=&\int_{|T_{z_{N,i}}(x-z_{N,i})|\leq s_{\delta_N,i}}+\int_{ s_{\delta_N,i}<|T_{z_{N,i}}(x-z_{N,i})|\leq \mu}+\int_{|T_{z_{N,i}}(x-z_{N,i})|> \mu}\\
&u_N \left( -\delta_N^2\text{div}\left( (K(x)-K(z_{N,i}))\nabla \frac{\partial V_{\delta_N, z_{N,i}, \hat{q}_{\delta_N,i}, z_{N,i}}}{\partial z_{i,\hbar}}\right) \right)+O\left( \frac{\delta_N^2}{|\ln\varepsilon_N|}\right)  \\
=&O\left( \frac{\delta_N^2}{|\ln\varepsilon_N|}\right) +O(\delta_N^2)+O\left( \frac{\delta_N^2}{|\ln\varepsilon_N|}\right)\\
=&O\left( \frac{\varepsilon_N^2}{|\ln\varepsilon_N|^{p-1}}\right).
\end{split}
\end{equation}
Here $ \mu>0 $ is a small constant, and we have used \eqref{200}, Lemmas \ref{lemA-2} and \ref{lemA-5}. Taking \eqref{302}, \eqref{303} and \eqref{304} into \eqref{301}, we get
\begin{equation*}
\begin{split}
&\int_{\Omega}u_NL_{\delta_N}\left( \frac{\partial V_{\delta_N, Z_N,i}}{\partial z_{i,\hbar}}\right) =O\left( \frac{\varepsilon_N\ln|\ln\varepsilon_N|}{|\ln\varepsilon_N|^{p+1}}\right) .
\end{split}
\end{equation*}
Combining with \eqref{coef of C} we get
\begin{equation*}
C_{j,h,N}=O(\varepsilon_N\ln|\ln\varepsilon_N|).
\end{equation*}
Hence
\begin{equation*}
\begin{split}
&\sum_{j=1}^m\sum_{h=1}^2C_{j,h,N}\frac{\partial}{\partial z_{j,h}}(-\delta_N^2\text{div}(K(z_{N,j})\nabla   V_{\delta_N, Z_N,j}))\\
=&p\sum_{j=1}^m\sum_{h=1}^2C_{j,h,N}(V_{\delta_N, z_{N,j}, \hat{q}_{\delta_N,j}, z_{N,j}}-\hat{q}_{\delta_N,j})^{p-1}_+\left( \frac{\partial V_{\delta_N, z_{N,j}, \hat{q}_{\delta_N,j}, z_{N,j}}}{\partial z_{j,h}}-\frac{\partial \hat{q}_{\delta_N,j}}{\partial z_{j,h}}\right) \\
=&O\left( \sum_{j=1}^m\sum_{h=1}^2\frac{\varepsilon_N^{\frac{2}{p}-1}|C_{j,h,N}|}{|\ln\varepsilon_N|^{p}}\right) +O\left( \sum_{j=1}^m\sum_{h=1}^2\frac{\varepsilon_N^{\frac{2}{p}}|C_{j,h,N}|}{|\ln\varepsilon_N|^{p}}\right) \\
=&O\left( \frac{\varepsilon_N^{\frac{2}{p}}\ln|\ln\varepsilon_N|}{|\ln\varepsilon_N|^{p}}\right),\ \ \ \ \text{in}\ L^p(\Omega).
\end{split}
\end{equation*}

So by the assumption and \eqref{208} we have
\begin{equation*}
\begin{split}
L_{\delta_N} u_N=&Q_{\delta_N} L_{\delta_N} u_N+\sum_{j=1}^m\sum_{h=1}^2C_{j,h,N}\frac{\partial}{\partial z_{j,h}}(-\delta_N^2\text{div}(K(z_{N,j})\nabla   V_{\delta_N, Z_N,j}))\\
=&O\left( \frac{1}{N}\frac{\varepsilon_N^{\frac{2}{p}}}{|\ln\varepsilon_N|^{p-1}}\right) +O\left( \frac{\varepsilon_N^{\frac{2}{p}}\ln|\ln\varepsilon_N|}{|\ln\varepsilon_N|^{p}}\right) \\
=&o\left( \frac{\varepsilon_N^{\frac{2}{p}}}{|\ln\varepsilon_N|^{p-1}}\right),\ \ \ \ \text{in}\ L^p(\Omega).
\end{split}
\end{equation*}

For any fixed $ i, $ define $ \tilde{u}_{N,i}(y)=u_N(s_{\delta_N, i}y+z_{N,i}) $ for $ y\in \Omega_{N,i}:=\{y\in\mathbb{R}^2\mid s_{\delta_N, i}y+z_{N,i}\in \Omega\} $.
Define $  $
\begin{equation*}
\tilde{L}_{N,i}u=-\text{div}(K(s_{\delta_N, i}y+z_{N,i})\nabla u)-p\frac{s_{\delta_N,i}^2}{\delta_N^2}(V_{\delta_N, Z_N}(s_{\delta_N, i}y+z_{N,i})-q(s_{\delta_N, i}y+z_{N,i}))^{p-1}_+u.
\end{equation*}
Then direct computation shows that
\begin{equation*}
\begin{split}
||\tilde{L}_{N,i}\tilde{u}_{N,i}||_{L^p(\Omega_{N,i})}
=&\left( \int_{\Omega_{N,i}}\left( \tilde{L}_{N,i}\tilde{u}_{N,i} \right)^pdy\right)^{\frac{1}{p}}  \\
=&\left( \frac{1}{s_{\delta_N, i}^2}\int_{\Omega}\left( -s_{\delta_N, i}^2\text{div}(K(x)\nabla u_N)-p\frac{s_{\delta_N, i}^2}{\delta_N^2}(V_{\delta_N,Z_N}-q)^{p-1}_+u_N \right)^pdx\right)^{\frac{1}{p}}\\
=&\frac{s_{\delta_N, i}^2}{s_{\delta_N, i}^{\frac{2}{p}}\delta_N^2}||L_{\delta_N}u_N||_{L^p(\Omega)},
\end{split}
\end{equation*}
i.e., $ s_{\delta_N, i}^{\frac{2}{p}}\cdot\frac{\delta_N^2}{s_{\delta_N, i}^2}||\tilde{L}_{N,i}\tilde{u}_{N,i}||_{L^p(\Omega_{N,i})}=||L_{\delta_N}u_N||_{L^p(\Omega)}. $

By the fact that $ \frac{\delta_N^2}{s_{\delta_N, i}^2}=O(\frac{1}{|\ln\varepsilon_N|^{p-1}}) $ and $ s_{\delta_N, i}=O(\varepsilon_N) $, we have
\begin{equation*}
\tilde{L}_{N,i}\tilde{u}_{N,i}=o(1)\ \ \ \  \text{in}\ \ L^p(\Omega_{N,i}).
\end{equation*}
Since $ ||\tilde{u}_{N,i}||_{L^\infty(\Omega_{N,i})}=1 $, by the classical regularity theory of  elliptic equations, $ \tilde{u}_{N,i} $ is uniformly bounded in $ W^{2,p}_{loc}(\mathbb{R}^2) $. Hence we may assume that
\begin{equation*}
\tilde{u}_{N,i}\to u_i\ \ \ \ \text{in}\ \ C^1_{loc}(\mathbb{R}^2).
\end{equation*}

We claim that $ u_i\equiv 0. $ On the one hand, by \eqref{203}, the definition of $ V_{\delta_N, z_{N,i}, \hat{q}_{\delta_N,i}, z_{N,i}} $ and the fact that $ z_{N,i}\to z_i $ as $ N\to\infty $, we obtain
\begin{equation*}
\begin{split}
&\frac{s_{\delta_N,i}^2}{\delta_N^2}(V_{\delta_N, Z_N}(s_{\delta_N, i}y+z_{N,i})-q(s_{\delta_N, i}y+z_{N,i}))^{p-1}_+\\
=&\frac{s_{\delta_N,i}^2}{\delta_N^2}\left( V_{\delta_N, z_{N,i}, \hat{q}_{\delta_N,i}, z_{N,i}}(s_{\delta_N, i}y+z_{N,i})-\hat{q}_{\delta_N,i}+O(\frac{\ln|\ln\varepsilon_N|}{|\ln\varepsilon_N|^2})\right)^{p-1}_+ \\
\to&\phi(T_{z_i}y)^{p-1}_+\ \  \ \ \text{in}\ C^0_{loc}(\mathbb{R}^2)\ \text{as}\ N\to\infty.
\end{split}
\end{equation*}
So $ u_i $ satisfies
\begin{equation*}
-\text{div}(K(z_i)\nabla u_i(x))-p\phi(T_{z_i}x)^{p-1}_+u_i(x)=0,\ \ x\in\mathbb{R}^2.
\end{equation*}
Let $ \hat{u}_i(x)=u_i(T_{z_i}^{-1}x) $. Note that $ T_{z_i}^{-1}(T_{z_i}^{-1})^t=K(z_i) $. Then
\begin{equation*}
-\Delta \hat{u}_i(x)=-\text{div}(K(z_i)\nabla u_i)(T_{z_i}^{-1}x)=p\phi(x)^{p-1}_+\hat{u}_i(x),\ \ \ \  \forall\ x\in \mathbb{R}^2.
\end{equation*}
Since $ \hat{u}_i(x)\in L^{\infty}(\mathbb{R}^2) $, it follows from Proposition \ref{Non-degenerate} that there are $ c_1,c_2 $ satisfying
\begin{equation}\label{210}
\hat{u}_i=c_1\frac{\partial \phi}{\partial x_1}+c_2\frac{\partial \phi}{\partial x_2}.
\end{equation}

On the other hand, since $ u_N\in E_{\delta_N,Z_N} $, we have
\begin{equation*}
\int_{\Omega}-\delta_N^2\text{div}\left( K(x)\nabla \frac{\partial V_{\delta_N, Z_N,i}}{\partial z_{i,h}}\right) u_N=0,\ \ \forall h=1,2,
\end{equation*}
which implies that
\begin{equation}\label{209}
\begin{split}
0=&\int_{\Omega}-\delta_N^2\text{div}\left( K(z_{N,i})\nabla \frac{\partial V_{\delta_N, Z_N,i}}{\partial z_{i,h}}\right) u_N+\int_{\Omega}-\delta_N^2\text{div}\left( (K(x)-K(z_{N,i}))\nabla \frac{\partial V_{\delta_N, Z_N,i}}{\partial z_{i,h}}\right) u_N\\
=&\int_{\Omega}\frac{\partial }{\partial z_{i,h}}\left[-\delta_N^2 \text{div}(K(z_{N,i})\nabla V_{\delta_N, Z_N,i})\right] u_N-\int_{\Omega}-\delta_N^2\text{div}\left( \frac{\partial K(z_{N,i})}{\partial z_{i,h}}\nabla V_{\delta_N, Z_N,i}\right) u_N\\
&+\int_{\Omega}-\delta_N^2\text{div}\left( (K(x)-K(z_{N,i}))\nabla \frac{\partial V_{\delta_N, Z_N,i}}{\partial z_{i,h}}\right) u_N.
\end{split}
\end{equation}
By \eqref{303}, we have
\begin{equation}\label{209-1}
\int_{\Omega}-\delta_N^2\text{div}\left( \frac{\partial K(z_{N,i})}{\partial z_{i,h}}\nabla V_{\delta_N, Z_N,i}\right) u_N=O\left( \frac{\varepsilon_N^2}{|\ln\varepsilon_N|^{p-1}}\right).
\end{equation}
By \eqref{304}, we get
\begin{equation}\label{209-2}
\int_{\Omega}-\delta_N^2\text{div}\left( (K(x)-K(z_{N,i}))\nabla \frac{\partial V_{\delta_N, Z_N,i}}{\partial z_{i,h}}\right) u_N=O\left( \frac{\varepsilon_N^2}{|\ln\varepsilon_N|^{p-1}}\right).
\end{equation}
Using the definition of $ V_{\delta_N, Z_N,i} $, \eqref{200-1} and \eqref{200}, we obtain
\begin{equation}\label{209-3}
\begin{split}
&\int_{\Omega}\frac{\partial }{\partial z_{i,h}}\left[-\delta_N^2 \text{div}(K(z_{N,i})\nabla V_{\delta_N, Z_N,i})\right] u_N\\
=&p\int_{\Omega}(V_{\delta_N, z_{N,i}, \hat{q}_{\delta_N,i}, z_{N,i}}-\hat{q}_{\delta_N,i})^{p-1}_+\left( \frac{\partial V_{\delta_N, z_{N,i}, \hat{q}_{\delta_N,i}, z_{N,i}}}{\partial z_{i,h}}-\frac{\partial \hat{q}_{\delta_N,i}}{\partial z_{i,h}}\right) u_N\\
=&p\int_{\Omega}\left( \frac{\delta_N}{s_{\delta_N,i}}\right)^2\phi\left( \frac{T_{z_{N,i}}(x-z_{N,i})}{s_{\delta_N,i}}\right)^{p-1}_+\frac{1}{s_{\delta_N,i}}\left( \frac{\delta_N}{s_{\delta_N,i}}\right)^{\frac{2}{p-1}}\phi'\left( \frac{T_{z_{N,i}}(x-z_{N,i})}{s_{\delta_N,i}}\right)\frac{(T_{z_{N,i}})_h^t\cdot T_{z_{N,i}}(x-z_{N,i})}{|T_{z_{N,i}}(x-z_{N,i})|}u_N\\
&+O\left( \frac{\varepsilon_N^2}{|\ln\varepsilon_N|^{p-1}}\right) \\
=&ps_{\delta_N,i}\left( \frac{\delta_N}{s_{\delta_N,i}}\right)^{\frac{2p}{p-1}}\int_{\mathbb{R}^2}\phi(T_{z_{N,i}}y)^{p-1}_+\phi'(T_{z_{N,i}}y)\frac{(T_{z_{N,i}})_h^t\cdot T_{z_{N,i}}y}{|T_{z_{N,i}}y|}\tilde{u}_{N,i}(y)dy+O\left( \frac{\varepsilon_N^2}{|\ln\varepsilon_N|^{p-1}}\right).
\end{split}
\end{equation}
Hence taking \eqref{209-1}, \eqref{209-2} and \eqref{209-3} into \eqref{209}, we have
\begin{equation}\label{209-0}
\begin{split}
0=&ps_{\delta_N,i}\left( \frac{\delta_N}{s_{\delta_N,i}}\right)^{\frac{2p}{p-1}}\int_{\mathbb{R}^2}\phi(T_{z_{N,i}}y)^{p-1}_+\phi'(T_{z_{N,i}}y)\frac{(T_{z_{N,i}})_h^t\cdot T_{z_{N,i}}y}{|T_{z_{N,i}}y|}\tilde{u}_{N,i}(y)dy+O\left( \frac{\varepsilon_N^2}{|\ln\varepsilon_N|^{p-1}}\right).
\end{split}
\end{equation}
Dividing both sides of \eqref{209-0} into $ ps_{\delta_N,i}(\frac{\delta_N}{s_{\delta_N,i}})^{\frac{2p}{p-1}} $ and passing $ N $ to the limit, we get for   $ h=1,2 $
\begin{equation*}
\begin{split}
0=&\int_{\mathbb{R}^2}\phi(T_{z_{i}}y)^{p-1}_+\phi'(T_{z_{i}}y)\frac{(T_{z_{i}})_h^t\cdot T_{z_{i}}y}{|T_{z_{i}}y|}u_{i}(y)dy\\
=&\int_{\mathbb{R}^2}\phi(x)^{p-1}_+\phi'(x)\frac{(T_{z_{i}})_h^t\cdot x}{|x|}\hat{u}_{i}(x)\sqrt{det(K(z_i))}dx,
\end{split}
\end{equation*}
which implies that
\begin{equation}\label{211}
0=\int_{B_1(0)}\phi^{p-1}_+\frac{\partial \phi}{\partial x_h}\hat{u}_i.
\end{equation}
Combining \eqref{210} with \eqref{211}, we have $ c_1=c_2=0. $ That is, $ u_i\equiv 0. $

So we conclude that $ \tilde{u}_{N,i}\to 0 $ in  $ C^1(B_L(0)) $, which implies that
\begin{equation}\label{212}
||u_N||_{L^\infty(B_{Ls_{\delta_N,i}}(z_{N,i}))}=o(1).
\end{equation}

By the assumption that $ Q_{\delta_N} L_{\delta_N} u_N = 0 $ in $ \Omega \backslash\cup_{i=1}^mB_{Ls_{\delta_N,i}}(z_{N,i}) $ and the definition of $ Q_{\delta_N} $, we have for $ L $ large
\begin{equation*}
 L_{\delta_N} u_N = 0\  \ \text{in} \ \Omega \backslash\cup_{i=1}^mB_{Ls_{\delta_N,i}}(z_{N,i}).
\end{equation*}
It follows from Lemma \ref{lemA-5} that
\begin{equation*}
(V_{\delta_N,Z_N}-q)_+=0\  \ \text{in} \ \Omega \backslash\cup_{i=1}^mB_{Ls_{\delta_N,i}}(z_{N,i}).
\end{equation*}
So we get $ -\text{div}(K(x)\nabla u_N)=0 $ in $ \Omega \backslash\cup_{i=1}^mB_{Ls_{\delta_N,i}}(z_{N,i}). $

Since $ u_N=o(1) $ on $ \cup_{i=1}^m\partial B_{Ls_{\delta_N,i}}(z_{N,i}) $ and $ u_N=0 $ on $ \partial \Omega $, by the maximum principle, we get
\begin{equation*}
||u_N||_{L^\infty(\Omega\backslash \cup_{i=1}^mB_{Ls_{\delta_N,i}}(z_{N,i}))}=o(1),
\end{equation*}
which combined with \eqref{212} we have
\begin{equation*}
||u_N||_{L^\infty(\Omega)}=o(1).
\end{equation*}
This is a contradiction since $ ||u_N||_{L^\infty(\Omega)}=1. $

\end{proof}

Then we can get
\begin{proposition}\label{one to one and onto}
$ Q_\delta L_\delta $ is  a one to one and onto map from $ E_{\delta,Z} $ to $ F_{\delta, Z}. $
\end{proposition}

\begin{proof}
If $ Q_\delta L_\delta u=0 $, by Lemma \ref{coercive esti}, $ u=0 $. So $ Q_\delta L_\delta $ is  a one to one  map from $ E_{\delta,Z} $ to $ F_{\delta, Z}. $

We prove that $ Q_\delta L_\delta $ is  an onto  map from $ E_{\delta,Z} $ to $ F_{\delta, Z}. $ Denote
\begin{equation*}
\hat{E}=\{u\in H^1_0(\Omega)\mid \int_{\Omega}\left( K(x)\nabla u|\nabla\frac{\partial V_{\delta,Z,i}}{\partial z_{i,h}}\right) =0,\ \ \ \ i=1,\cdots,m,\ h=1,2\}.
\end{equation*}
Then $ E_{\delta,Z}=\hat{E}\cap W^{2,p}(\Omega) $. For any $ \hat{h}\in F_{\delta,Z}  $, by the Riesz representation theorem there is a unique $ u\in H^1_0(\Omega) $ such that
\begin{equation}\label{213}
\delta^2\int_{\Omega}\left( K(x)\nabla u|\nabla\varphi\right) =\int_{\Omega}\hat{h}\varphi,\ \ \   \ \forall \varphi\in H^1_0(\Omega).
\end{equation}

Since $ \hat{h}\in F_{\delta,Z} $, we have $ u\in \hat{E}. $ Using the classical $ L^p $ theory, we conclude that $ u\in W^{2,p}(\Omega) $, which implies that $ u\in E_{\delta,Z}. $ Thus $ -\delta^2\text{div}(K(x)\nabla)=Q_{\delta}(-\delta^2\text{div}(K(x)\nabla)) $ is a one to one and onto map from $ E_{\delta,Z} $ to $ F_{\delta, Z}. $

For any $ h\in F_{\delta,Z} $,  $ Q_\delta L_\delta u=h $ is equivalent to
\begin{equation}\label{214}
u=(Q_{\delta}(-\delta^2\text{div}(K(x)\nabla)))^{-1}pQ_\delta(V_{\delta,Z}-q)^{p-1}_+u+(Q_{\delta}(-\delta^2\text{div}(K(x)\nabla)))^{-1}h,\ \ u\in E_{\delta,Z}.
\end{equation}
Note that $\mathcal{T}u:= (Q_{\delta}(-\delta^2\text{div}(K(x)\nabla)))^{-1}pQ_\delta(V_{\delta,Z}-q)^{p-1}_+u $ is a compact operator in $ E_{\delta,Z}. $ By the Fredholm alternative, \eqref{214} is solvable if and only if
\begin{equation*}
u=(Q_{\delta}(-\delta^2\text{div}(K(x)\nabla)))^{-1}pQ_\delta(V_{\delta,Z}-q)^{p-1}_+u
\end{equation*}
has only trivial solution, which is true since $ Q_\delta L_\delta  $ is one to one. The proof is thus complete.
\end{proof}

Now  consider solutions of \eqref{eq1} being the form of $ V_{\delta,Z}+\omega_\delta $. Note that by \eqref{eq1}, one has
\begin{equation*}
-\delta^2\text{div}(K(x)\nabla(V_{\delta,Z}+\omega_\delta))-(V_{\delta,Z}+\omega_\delta-q)^p_+=0,
\end{equation*}
which is equivalent to
\begin{equation}\label{215}
\begin{split}
L_\delta \omega_\delta=l_{1,\delta}+l_{2,\delta}+R_\delta(\omega_\delta),
\end{split}
\end{equation}
where
\begin{equation*}
l_{1,\delta}=(V_{\delta,Z}-q)^p_+-\sum_{j=1}^m(V_{\delta,z_j,\hat{q}_{\delta,j}, z_j}-\hat{q}_{\delta,j})^p_+,
\end{equation*}
\begin{equation*}
l_{2,\delta}=\delta^2\sum_{j=1}^m\text{div}((K(x)-K(z_j))\nabla V_{\delta,Z,j}),
\end{equation*}
\begin{equation*}
R_\delta(\omega_\delta)=(V_{\delta,Z}+\omega_\delta-q)^p_+-(V_{\delta,Z}-q)^p_+-p(V_{\delta,Z}-q)^{p-1}_+\omega_\delta.
\end{equation*}

We first solve the existence and uniqueness of $ \omega\in E_{\delta,Z}  $ satisfying
\begin{equation}\label{216}
Q_\delta L_\delta \omega=Q_\delta l_{1,\delta}+Q_\delta l_{2,\delta}+Q_\delta R_\delta(\omega),
\end{equation}
or equivalently,
\begin{equation*}
\begin{split}
\omega=T_\delta(\omega):=(Q_\delta L_\delta)^{-1}Q_\delta l_{1,\delta}+(Q_\delta L_\delta)^{-1}Q_\delta l_{2,\delta}+(Q_\delta L_\delta)^{-1}Q_\delta R_\delta(\omega).
\end{split}
\end{equation*}
And then we can reduce \eqref{215} to a finite dimensional problem.

\begin{proposition}\label{exist and uniq of w}
There is $ \delta_0>0, $ such that for any $ 0<\delta<\delta_0 $ and $ Z $ satisfying \eqref{admis set}, \eqref{216} has the unique solution $ \omega_{\delta,Z}\in E_{\delta,Z} $ with
\begin{equation*}
||\omega_{\delta,Z}||_{L^\infty(\Omega)}=O\left( \frac{\ln|\ln\varepsilon|}{|\ln\varepsilon|^2}\right).
\end{equation*}
\end{proposition}

\begin{proof}
It follows from  Lemma \ref{lemA-5} in Appendix that for $ L $ sufficiently large and $ \delta $ small,
\begin{equation*}
(V_{\delta,Z}-q)_+=0,\ \ \ \ \text{in}\ \Omega \backslash\cup_{i=1}^mB_{Ls_{\delta,i}}(z_{i}).
\end{equation*}
Let $ N= E_{\delta,Z} \cap\{\omega\mid ||\omega||_{L^\infty(\Omega)}\leq \frac{1}{|\ln\varepsilon|^{2-\theta_0}}\}$ for some $ \theta_0\in(0,1). $ Then $ N $ is complete under $ L^\infty $ norm and $ T_\delta $ is a map from $ E_{\delta,Z} $ to $ E_{\delta,Z} $. We now prove that $ T_\delta $ is a contraction map from $ N $ to $ N $.

First, we claim that $ T_\delta $ is a map from $ N $ to $ N $. For any $ \omega\in N $, by Lemma \ref{lemA-5} we get that for $ L>1 $ large and  $ \delta $ small,
\begin{equation*}
(V_{\delta,Z}+\omega-q)_+=0,\ \ \ \ \text{in}\ \Omega \backslash\cup_{i=1}^mB_{Ls_{\delta,i}}(z_{i}).
\end{equation*}
So $ l_{1,\delta}=R_\delta(\omega)=0 $ in $ \Omega \backslash\cup_{i=1}^mB_{Ls_{\delta,i}}(z_{i}). $ Note that for any $ u\in L^\infty(\Omega) $,
\begin{equation*}
Q_\delta u=u,\ \ \ \ \text{in}\ \Omega \backslash\cup_{i=1}^mB_{Ls_{\delta,i}}(z_{i}).
\end{equation*}
Hence
\begin{equation*}
Q_\delta l_{1,\delta}+Q_\delta R_\delta(\omega)=0,\ \ \ \ \text{in}\ \Omega \backslash\cup_{i=1}^mB_{Ls_{\delta,i}}(z_{i}).
\end{equation*}
So, we can apply Lemma \ref{coercive esti} to obtain
\begin{equation*}
||(Q_\delta L_\delta)^{-1}(Q_\delta l_{1,\delta}+Q_\delta R_\delta(\omega))||_{L^\infty}\leq C\frac{|\ln\varepsilon|^{p-1}}{\varepsilon^{\frac{2}{p}}}||Q_\delta l_{1,\delta}+Q_\delta R_\delta(\omega)||_{L^p}.
\end{equation*}
By Lemma \ref{lemA-4}, we know that
\begin{equation*}
||Q_\delta l_{1,\delta}+Q_\delta R_\delta(\omega)||_{L^p}\leq C(|| l_{1,\delta}||_{L^p}+ ||R_\delta(\omega)||_{L^p}).
\end{equation*}
It follows from \eqref{203}, the definition of $ l_{1,\delta}, R_\delta(\omega) $ and Lemma \ref{lemA-5} that
\begin{equation*}
\begin{split}
||l_{1,\delta}||_{L^p}=&||(V_{\delta,Z}-q)^p_+-\sum_{j=1}^m(V_{\delta,z_j,\hat{q}_{\delta,j}, z_j}-\hat{q}_{\delta,j})^p_+||_{L^p}\\
\leq& C\frac{\ln|\ln\varepsilon|}{|\ln\varepsilon|^2}\sum_{j=1}^m||(V_{\delta,z_j,\hat{q}_{\delta,j}, z_j}-\hat{q}_{\delta,j})^{p-1}_+||_{L^p}\\
\leq&C\frac{\ln|\ln\varepsilon|}{|\ln\varepsilon|^2}\sum_{j=1}^m\left( \frac{\delta}{s_{\delta,j}}\right)^2s_{\delta,j}^{\frac{2}{p}}\\
\leq&C\frac{\varepsilon^{\frac{2}{p}}\ln|\ln\varepsilon|}{|\ln\varepsilon|^{p+1}},
\end{split}
\end{equation*}
and
\begin{equation*}
\begin{split}
||R_\delta(\omega)||_{L^p}=&||(V_{\delta,Z}+\omega-q)^p_+-(V_{\delta,Z}-q)^p_+-p(V_{\delta,Z}-q)^{p-1}_+\omega||_{L^p}\\
\leq &C||(V_{\delta,Z}-q)^{p-2}_+||_{L^p}||\omega||_{L^\infty}^2\\
\leq &C\frac{\varepsilon^{\frac{2}{p}}}{|\ln\varepsilon|^{p-2}}||\omega||_{L^\infty}^2.
\end{split}
\end{equation*}
So 
\begin{equation*}
\begin{split}
||(Q_\delta L_\delta)^{-1}(Q_\delta l_{1,\delta}+Q_\delta R_\delta(\omega))||_{L^\infty}\leq&C\varepsilon^{-\frac{2}{p}}|\ln\varepsilon|^{p-1}(|| l_{1,\delta}||_{L^p}+ ||R_\delta(\omega)||_{L^p})\\
\leq&C\varepsilon^{-\frac{2}{p}}|\ln\varepsilon|^{p-1}\left( \frac{\varepsilon^{\frac{2}{p}}\ln|\ln\varepsilon|}{|\ln\varepsilon|^{p+1}}+\frac{\varepsilon^{\frac{2}{p}}}{|\ln\varepsilon|^{p-2}}||\omega||_{L^\infty}^2\right).
\end{split}
\end{equation*}
By the definition of $ l_{2,\delta} $, one computes directly that  
$$ ||(Q_\delta L_\delta)^{-1}Q_\delta l_{2,\delta}||_{L^\infty}\leq C ||Q_\delta l_{2,\delta}||_{L^\infty}\leq \frac{C}{|\ln\delta|^2}.$$
Hence by the definition of $ N $,
\begin{equation}\label{217}
\begin{split}
||T_\delta(\omega)||_{L^\infty}\leq& C\varepsilon^{-\frac{2}{p}}|\ln\varepsilon|^{p-1}(|| l_\delta||_{L^p}+ ||R_\delta(\omega)||_{L^p})+\frac{C}{|\ln\delta|^2}\\
\leq&C\varepsilon^{-\frac{2}{p}}|\ln\varepsilon|^{p-1}\left( \frac{\varepsilon^{\frac{2}{p}}\ln|\ln\varepsilon|}{|\ln\varepsilon|^{p+1}}+\frac{\varepsilon^{\frac{2}{p}}}{|\ln\varepsilon|^{p-2}}||\omega||_{L^\infty}^2\right)+\frac{C}{|\ln\delta|^2} \\
\leq &\frac{1}{|\ln\varepsilon|^{2-\theta_0}}.
\end{split}
\end{equation}
So $ T_\delta $ is a map from $ N $ to $ N $.

Then we prove that  $ T_\delta $ is a contraction map. For any $ \omega_1,\omega_2\in N $,
\begin{equation*}
T_\delta(\omega_1)-T_\delta(\omega_2)=(Q_\delta L_\delta)^{-1}Q_\delta(R_\delta(\omega_1)-R_\delta(\omega_2)).
\end{equation*}
Note that $ R_\delta(\omega_1)=R_\delta(\omega_2)=0 $ in $ \Omega \backslash\cup_{i=1}^mB_{Ls_{\delta,i}}(z_{i}). $ By Lemma \ref{coercive esti} and the definition of $ N $, for $ \delta $ sufficiently small
\begin{equation*}
\begin{split}
||T_\delta(\omega_1)-T_\delta(\omega_2)||_{L^\infty}\leq &C\varepsilon^{-\frac{2}{p}}|\ln\varepsilon|^{p-1}||R_\delta(\omega_1)-R_\delta(\omega_2)||_{L^p}\\
\leq &C\varepsilon^{-\frac{2}{p}}|\ln\varepsilon|^{p-1}\varepsilon^{\frac{2}{p}}\left(\frac{||\omega_1||_{L^\infty}+||\omega_2||_{L^\infty}}{|\ln\varepsilon|^{p-2}} \right) ||\omega_1-\omega_2||_{L^\infty}\\
\leq& \frac{1}{2}||\omega_1-\omega_2||_{L^\infty}.
\end{split}
\end{equation*}
So $ T_\delta $ is a contraction map.

To conclude,  $ T_\delta $ is a contraction map from $ N $ to $ N $ and thus there is a unique $ \omega_{\delta,Z}\in N $ such that $ \omega_{\delta,Z}=T_\delta(\omega_{\delta,Z}) $. Moreover, from \eqref{217}, we have $ ||\omega_{\delta,Z}||_{L^\infty(\Omega)}=O\left( \frac{\ln|\ln\varepsilon|}{|\ln\varepsilon|^2}\right).  $

\end{proof}

\begin{remark}\label{rk1}
Indeed, since $ K,q \in C^\infty$ and $ p>1 $,  we can also check that $ \omega_{\delta,Z} $ is a $ C^1 $ map about $ Z $, see \cite{CLW,CPY} for example.
\end{remark}

\section{Proof of Theorem \ref{thm1}}

From Proposition \ref{exist and uniq of w}, we know that for any $ \delta $ small and $ Z $ satisfying \eqref{admis set}, there exists the unique $ \omega_{\delta,Z}\in E_{\delta,Z} $ satisfying
\begin{equation*}
Q_\delta L_\delta \omega_{\delta,Z}=Q_\delta l_\delta+Q_\delta R_\delta(\omega_{\delta,Z}),
\end{equation*}
i.e., for some $ C_{j,h} $ one has
\begin{equation*}
 L_\delta \omega_{\delta,Z}= l_\delta+  R_\delta(\omega_{\delta,Z})+\sum_{j=1}^m\sum_{h=1}^2C_{j,h}\frac{\partial}{\partial z_{j,h}}(-\delta^2\text{div}(K(z_j)\nabla   V_{\delta, Z,j})).
\end{equation*}

In the following we find proper $ Z=Z(\delta) $ such that $ V_{\delta,Z}+\omega_{\delta,Z} $ is a solution of \eqref{eq1}. Note that the associated functional of \eqref{eq1} is
\begin{equation}\label{functional I}
I_\delta(u)=\frac{\delta^2}{2}\int_{\Omega}\left( K(x)\nabla u|\nabla u\right)-\frac{1}{p+1}\int_{\Omega}(u-q)^{p+1}_+.
\end{equation}
Denote
\begin{equation*}
P_\delta(Z)=I_\delta(V_{\delta,Z}+\omega_{\delta,Z})
\end{equation*}
Then by Remark \ref{rk1}, we know that $ P_\delta(Z) $ is a $ C^1 $ function.

By the regularity of $ K $ and $ q $, it is not hard to check that if $ Z $ is a critical point of $ P_\delta  $, then $ V_{\delta,Z}+\omega_{\delta,Z} $ is a critical point of $ I_\delta $, i.e., a solution of \eqref{eq1}. We now prove that $ P_\delta $ has a critical point.

\begin{proposition}\label{pro401}
There holds
\begin{equation*}
P_{\delta}(Z)=I_\delta(V_{\delta,Z})+O\left( \frac{\varepsilon^2\ln|\ln\varepsilon|}{|\ln\varepsilon|^{p+2}}\right).
\end{equation*}
\end{proposition}

\begin{proof}
Note  that
\begin{equation*}
\begin{split}
P_{\delta}(Z)=&I_\delta(V_{\delta,Z})+\delta^2\int_{\Omega}\left( K(x)\nabla V_{\delta,Z}|\nabla \omega_{\delta,Z}\right) +\frac{\delta^2}{2}\int_{\Omega}\left( K(x)\nabla \omega_{\delta,Z}|\nabla \omega_{\delta,Z}\right)\\
&-\frac{1}{p+1}\left( \int_{\Omega}(V_{\delta,Z}+\omega_{\delta,Z}-q)^{p+1}_+-\int_{\Omega}(V_{\delta,Z}-q)^{p+1}_+\right).
\end{split}
\end{equation*}
By Proposition \ref{exist and uniq of w}, we get
\begin{equation*}
\begin{split}
&\int_{\Omega}(V_{\delta,Z}+\omega_{\delta,Z}-q)^{p+1}_+-\int_{\Omega}(V_{\delta,Z}-q)^{p+1}_+\\
=&(p+1)\sum_{j=1}^m\int_{B_{Ls_{\delta,j}}(z_j)}(V_{\delta,Z}-q)^{p}_+\omega_{\delta,Z}+O\left( \sum_{j=1}^m\int_{B_{Ls_{\delta,j}}(z_j)}(V_{\delta,Z}-q)^{p-1}_+\omega_{\delta,Z}^2\right) \\
=&O\left( \sum_{j=1}^m\frac{s_{\delta,j}^2||\omega_{\delta,Z}||_{L^\infty}}{|\ln\varepsilon|^p}\right) +O\left( \sum_{j=1}^m\frac{s_{\delta,j}^2||\omega_{\delta,Z}||_{L^\infty}^2}{|\ln\varepsilon|^{p-1}}\right) \\
=&O\left( \frac{\varepsilon^2\ln|\ln\varepsilon|}{|\ln\varepsilon|^{p+2}}\right).
\end{split}
\end{equation*}
Using Proposition \ref{exist and uniq of w}, we obtain
\begin{equation*}
\begin{split}
&\delta^2\int_{\Omega}\left( K(x)\nabla V_{\delta,Z}|\nabla \omega_{\delta,Z}\right)\\
=&\sum_{j=1}^m\int_{B_{Ls_{\delta,j}}(z_j)}(V_{\delta,z_j,\hat{q}_{\delta,j},z_j}-\hat{q}_{\delta,j})^p_+\omega_{\delta,Z}+\sum_{j=1}^m\delta^2\int_{\Omega}\left( (K(x)-K(z_j))\nabla V_{\delta,Z}|\nabla \omega_{\delta,Z}\right)\\
=&O\left( \sum_{j=1}^m\frac{s_{\delta,j}^2||\omega_{\delta,Z}||_{L^\infty}}{|\ln\varepsilon|^p}\right) +O\left( \frac{\delta^2||\omega_{\delta,Z}||_{L^\infty}}{|\ln\varepsilon|}\right)\\
=&O\left( \frac{\varepsilon^2\ln|\ln\varepsilon|}{|\ln\varepsilon|^{p+2}}\right).
\end{split}
\end{equation*}
Now we calculate the term $ \frac{\delta^2}{2}\int_{\Omega}\left( K(x)\nabla \omega_{\delta,Z}|\nabla \omega_{\delta,Z}\right). $ Since  $ \omega_{\delta,Z}\in E_{\delta,Z} $, we get
\begin{equation*}
\begin{split}
Q_\delta L_\delta \omega_{\delta,Z}=&Q_\delta(-\delta^2\text{div}(K(x)\nabla \omega_{\delta,Z}))-Q_\delta(p(V_{\delta,Z}-q)^{p-1}_+\omega_{\delta,Z})\\
=&-\delta^2\text{div}(K(x)\nabla \omega_{\delta,Z})-Q_\delta(p(V_{\delta,Z}-q)^{p-1}_+\omega_{\delta,Z}),
\end{split}
\end{equation*}
which combined with $ Q_\delta L_\delta \omega_{\delta,Z}=Q_\delta l_\delta+Q_\delta R_\delta(\omega_{\delta,Z})  $ yields
\begin{equation*}
-\delta^2\text{div}(K(x)\nabla \omega_{\delta,Z})=Q_\delta(p(V_{\delta,Z}-q)^{p-1}_+\omega_{\delta,Z})+Q_\delta l_\delta+Q_\delta R_\delta(\omega_{\delta,Z}).
\end{equation*}
Hence by Lemmas \ref{lemA-4}, \ref{lemA-5} and Proposition \ref{exist and uniq of w}, we get
\begin{equation*}
\begin{split}
&\delta^2\int_{\Omega}\left( K(x)\nabla \omega_{\delta,Z}|\nabla \omega_{\delta,Z}\right)\\
=&\int_{\Omega}Q_\delta(p(V_{\delta,Z}-q)^{p-1}_+\omega_{\delta,Z})\omega_{\delta,Z}+\int_{\Omega}Q_\delta l_\delta\omega_{\delta,Z}+\int_{\Omega}Q_\delta R_\delta(\omega_{\delta,Z})\omega_{\delta,Z}\\
\leq &||Q_\delta(p(V_{\delta,Z}-q)^{p-1}_+\omega_{\delta,Z})||_{L^1}||\omega_{\delta,Z}||_{L^\infty}+||Q_\delta l_\delta||_{L^1}||\omega_{\delta,Z}||_{L^\infty}+||Q_\delta R_\delta(\omega_{\delta,Z})||_{L^1}||\omega_{\delta,Z}||_{L^\infty}\\
\leq&C(||(p(V_{\delta,Z}-q)^{p-1}_+\omega_{\delta,Z})||_{L^1}+|| l_\delta||_{L^1}+|| R_\delta(\omega_{\delta,Z})||_{L^1})||\omega_{\delta,Z}||_{L^\infty}\\
= &O\left( \frac{\varepsilon^2||\omega_{\delta,Z}||_{L^\infty}^2}{|\ln\varepsilon|^{p-1}}\right) +O\left( \frac{\varepsilon^2\ln|\ln\varepsilon|||\omega_{\delta,Z}||_{L^\infty}}{|\ln\varepsilon|^{p+1}}\right) +O\left( \frac{\varepsilon^2||\omega_{\delta,Z}||_{L^\infty}^3}{|\ln\varepsilon|^{p-2}}\right) \\
=&O\left( \frac{\varepsilon^2\ln|\ln\varepsilon|}{|\ln\varepsilon|^{p+2}}\right).
\end{split}
\end{equation*}
To conclude, we get $ P_{\delta}(Z)=I_\delta(V_{\delta,Z})+O\left(\frac{\varepsilon^2\ln|\ln\varepsilon|}{|\ln\varepsilon|^{p+2}}\right). $

\end{proof}

\begin{lemma}\label{order of main term}
There holds
\begin{equation*}
I_\delta(V_{\delta,Z})=\sum_{j=1}^m\frac{\pi\delta^2}{\ln\frac{R}{\varepsilon}}q^2(z_j)\sqrt{det(K(z_j))}+O\left( \frac{\delta^2\ln|\ln\varepsilon|}{|\ln\varepsilon|^2}\right).
\end{equation*}
\end{lemma}
\begin{proof}
Note that  by the definition of $ V_{\delta,Z} $,
\begin{equation}\label{218}
\begin{split}
I_\delta(V_{\delta,Z})=&\frac{\delta^2}{2}\int_{\Omega}\left( K(x)\nabla V_{\delta,Z}|\nabla V_{\delta,Z}\right)-\frac{1}{p+1}\int_{\Omega}(V_{\delta,Z}-q)^{p+1}_+\\
=&\frac{1}{2}\sum_{j=1}^m\int_{\Omega}(V_{\delta,z_j,\hat{q}_{\delta,j},z_j}-\hat{q}_{\delta,j})^p_+V_{\delta,Z,j}+\frac{1}{2}\sum_{j=1}^m\delta^2\int_{\Omega}\left( (K(x)-K(z_j))\nabla V_{\delta,Z,j}|\nabla V_{\delta,Z,j}\right)\\
&+\frac{1}{2}\sum_{1\leq i\neq j\leq m}\int_{\Omega}(V_{\delta,z_j,\hat{q}_{\delta,j},z_j}-\hat{q}_{\delta,j})^p_+V_{\delta,Z,i}+\frac{1}{2}\sum_{1\leq i\neq j\leq m}\delta^2\int_{\Omega}\left( (K(x)-K(z_j))\nabla V_{\delta,Z,j}|\nabla V_{\delta,Z,i}\right)\\
&-\frac{1}{p+1}\int_{\Omega}(V_{\delta,Z}-q)^{p+1}_+.
\end{split}
\end{equation}
By \eqref{202}, we have
\begin{equation*}
\begin{split}
&\int_{\Omega}(V_{\delta,z_j,\hat{q}_{\delta,j},z_j}-\hat{q}_{\delta,j})^p_+V_{\delta,Z,j}\\
=&\hat{q}_{\delta,j}\int_{\Omega}(V_{\delta,z_j,\hat{q}_{\delta,j},z_j}-\hat{q}_{\delta,j})^p_++\int_{\Omega}(V_{\delta,z_j,\hat{q}_{\delta,j},z_j}-\hat{q}_{\delta,j})^{p+1}_+-\frac{\hat{q}_{\delta,j}}{\ln\frac{R}{s_{\delta,j}}}\int_{\Omega}(V_{\delta,z_j,\hat{q}_{\delta,j},z_j}-\hat{q}_{\delta,j})^p_+g_{z_j}(T_{z_j}x,T_{z_j}z_j).
\end{split}
\end{equation*}
By the definition of $ V_{\delta,z_j,\hat{q}_{\delta,j},z_j} $, the fact that $ T_{z_j}^{-1}(T_{z_j}^{-1})^t=K(z_j) $ and \eqref{PI}, we get
\begin{equation*}
\begin{split}
&\hat{q}_{\delta,j}\int_{\Omega}(V_{\delta,z_j,\hat{q}_{\delta,j},z_j}-\hat{q}_{\delta,j})^p_+\\
=&\hat{q}_{\delta,j} s_{\delta,j}^2(\frac{\delta}{s_{\delta,j}})^{\frac{2p}{p-1}}\int_{|T_{z_j}x|\leq 1}\phi(T_{z_j}x)^pdx\\
=&\hat{q}_{\delta,j} s_{\delta,j}^2(\frac{\delta}{s_{\delta,j}})^{\frac{2p}{p-1}}\sqrt{det(K(z_j))}\cdot 2\pi|\phi'(1)|\\
=&\hat{q}_{\delta,j}\delta^2|\phi'(1)|^{p-1}\left( \frac{\ln\frac{R}{s_{\delta,j}}}{\hat{q}_{\delta,j}}\right)^{p-1}|\phi'(1)|^{-p}\left( \frac{\ln\frac{R}{s_{\delta,j}}}{\hat{q}_{\delta,j}}\right)^{-p}\sqrt{det(K(z_j))}\cdot 2\pi|\phi'(1)|\\
=&\frac{2\pi\delta^2}{\ln\frac{R}{s_{\delta,j}}}\hat{q}_{\delta,j}^2\sqrt{det(K(z_j))}.
\end{split}
\end{equation*}
Similarly, we have
\begin{equation*}
\begin{split}
&\int_{\Omega}(V_{\delta,z_j,\hat{q}_{\delta,j},z_j}-\hat{q}_{\delta,j})^{p+1}_+\\
=& s_{\delta,j}^2\left( \frac{\delta}{s_{\delta,j}}\right)^{\frac{2(p+1)}{p-1}}\sqrt{det(K(z_j))}\cdot \frac{(p+1)\pi}{2}|\phi'(1)|^2\\
=&\delta^2|\phi'(1)|^{p-1}\left( \frac{\ln\frac{R}{s_{\delta,j}}}{\hat{q}_{\delta,j}}\right)^{p-1}|\phi'(1)|^{-(p+1)}\left( \frac{\ln\frac{R}{s_{\delta,j}}}{\hat{q}_{\delta,j}}\right)^{-(p+1)}\sqrt{det(K(z_j))}\cdot \frac{(p+1)\pi}{2}|\phi'(1)|^2\\
=&\frac{(p+1)\pi\delta^2}{2\left(  \ln\frac{R}{s_{\delta,j}}\right)^2}\hat{q}_{\delta,j}^2\sqrt{det(K(z_j))},
\end{split}
\end{equation*}
and
\begin{equation*}
\begin{split}
&\frac{\hat{q}_{\delta,j}}{\ln\frac{R}{s_{\delta,j}}}\int_{\Omega}(V_{\delta,z_j,\hat{q}_{\delta,j},z_j}-\hat{q}_{\delta,j})^p_+g_{z_j}(T_{z_j}x,T_{z_j}z_j)\\
=&g_{z_j}(T_{z_j}z_j,T_{z_j}z_j)\frac{\hat{q}_{\delta,j}}{\ln\frac{R}{s_{\delta,j}}}\int_{\Omega}(V_{\delta,z_j,\hat{q}_{\delta,j},z_j}-\hat{q}_{\delta,j})^p_++O\left( \frac{\varepsilon^3|\nabla g_{z_j}(T_{z_j}z_j,T_{z_j}z_j)|}{|\ln\varepsilon|^{p+1}}\right) \\
=&\frac{2\pi\delta^2g_{z_j}(T_{z_j}z_j,T_{z_j}z_j)}{\left( \ln\frac{R}{s_{\delta,j}}\right)^2}\hat{q}_{\delta,j}^2\sqrt{det(K(z_j))}+O\left(\frac{\varepsilon^3}{|\ln\varepsilon|^{p+1}}\right),
\end{split}
\end{equation*}
where we have used  Lemma \ref{lemA-2}. Thus we get
\begin{equation}\label{219}
\begin{split}
\int_{\Omega}(V_{\delta,z_j,\hat{q}_{\delta,j},z_j}-\hat{q}_{\delta,j})^p_+V_{\delta,Z,j}
=&\frac{2\pi\delta^2}{\ln\frac{R}{s_{\delta,j}}}\hat{q}_{\delta,j}^2\sqrt{det(K(z_j))}+\frac{(p+1)\pi\delta^2}{2\left(  \ln\frac{R}{s_{\delta,j}}\right)^2}\hat{q}_{\delta,j}^2\sqrt{det(K(z_j))}\\
&-\frac{2\pi\delta^2g_{z_j}(T_{z_j}z_j,T_{z_j}z_j)}{\left( \ln\frac{R}{s_{\delta,j}}\right)^2}\hat{q}_{\delta,j}^2\sqrt{det(K(z_j))}+O\left( \frac{\varepsilon^3}{|\ln\varepsilon|^{p+1}}\right).
\end{split}
\end{equation}
Similarly for $ 1\leq i\neq j\leq m $,
\begin{equation}\label{220}
\begin{split}
&\int_{\Omega}(V_{\delta,z_j,\hat{q}_{\delta,j},z_j}-\hat{q}_{\delta,j})^p_+V_{\delta,Z,i}\\
=&\frac{\hat{q}_{\delta,i}}{\ln\frac{R}{s_{\delta,i}}}\int_{\Omega}(V_{\delta,z_j,\hat{q}_{\delta,j},z_j}-\hat{q}_{\delta,j})^p_+\bar{G}_{z_i}(T_{z_i}x,T_{z_i}z_i)\\
=&\frac{2\pi\delta^2\bar{G}_{z_i}(T_{z_i}z_j,T_{z_i}z_i)}{\ln\frac{R}{s_{\delta,i}}\ln\frac{R}{s_{\delta,j}}}\hat{q}_{\delta,i}\hat{q}_{\delta,j}\sqrt{det(K(z_j))}+O\left(\frac{\varepsilon^3}{|\ln\varepsilon|^{p+1}}\right).
\end{split}
\end{equation}
For any $ 1\leq j\leq m $, using \eqref{200} and Lemma \ref{lemA-2} we have
\begin{equation}\label{221}
\begin{split}
&\delta^2\int_{\Omega}\left( (K(x)-K(z_j))\nabla V_{\delta,Z,j}|\nabla V_{\delta,Z,j}\right)\\
=&\int_{T_{z_j}^{-1}(B_{s_{\delta,j}}(0))+z_j}+\int_{T_{z_j}^{-1}(B_{\mu}(0))+z_j\backslash T_{z_j}^{-1}(B_{s_{\delta,j}}(0))+z_j}+\int_{\Omega\backslash T_{z_j}^{-1}(B_{\mu}(0))+z_j}\\
&\delta^2\left( (K(x)-K(z_j))\nabla V_{\delta,Z,j}|\nabla V_{\delta,Z,j}\right)\\
=&O\left( \frac{\delta^2\varepsilon}{|\ln\varepsilon|^2}\right) +O\left( \frac{\delta^2}{|\ln\varepsilon|^2}\right) +O\left( \frac{\delta^2}{|\ln\varepsilon|^2}\right) \\
=&O\left( \frac{\delta^2}{|\ln\varepsilon|^2}\right).
\end{split}
\end{equation}
Similarly for  $ 1\leq i\neq j\leq m $,
\begin{equation}\label{222}
\begin{split}
\delta^2\int_{\Omega}\left( (K(x)-K(z_j))\nabla V_{\delta,Z,j}|\nabla V_{\delta,Z,i}\right)
=O\left( \frac{\delta^2}{|\ln\varepsilon|^2}\right).
\end{split}
\end{equation}
Finally by Lemma \ref{lemA-5} and \eqref{203}, we get
\begin{equation}\label{223}
\begin{split}
&\int_{\Omega}(V_{\delta,Z}-q)^{p+1}_+\\
=&\sum_{j=1}^m\int_{B_{Ls_{\delta,j}}(z_j)}\left( V_{\delta, z_j, \hat{q}_{\delta,j}, z_j}-\hat{q}_{\delta,j}+O\left( \frac{\ln|\ln\varepsilon|}{|\ln\varepsilon|^2}\right) \right)^{p+1}_+\\
=&\sum_{j=1}^m\int_{\Omega}(V_{\delta, z_j, \hat{q}_{\delta,j}, z_j}-\hat{q}_{\delta,j})^{p+1}_++O\left( \frac{\ln|\ln\varepsilon|}{|\ln\varepsilon|^2}\sum_{j=1}^m\int_{\Omega}(V_{\delta, z_j, \hat{q}_{\delta,j}, z_j}-\hat{q}_{\delta,j})^{p}_+\right) \\
=&\sum_{j=1}^m\frac{(p+1)\pi\delta^2}{2\left(  \ln\frac{R}{s_{\delta,j}}\right)^2}\hat{q}_{\delta,j}^2\sqrt{det(K(z_j))}+O\left( \frac{\delta^2\ln|\ln\varepsilon|}{|\ln\varepsilon|^3}\right).
\end{split}
\end{equation}
Taking \eqref{219}, \eqref{220}, \eqref{221}, \eqref{222} and \eqref{223} into \eqref{218}, one has
\begin{equation*}
\begin{split}
I_\delta(V_{\delta,Z})=&\sum_{j=1}^m\frac{\pi\delta^2}{\ln\frac{R}{s_{\delta,j}}}\hat{q}_{\delta,j}^2\sqrt{det(K(z_j))}+\sum_{j=1}^m\frac{(p+1)\pi\delta^2}{4\left(  \ln\frac{R}{s_{\delta,j}}\right)^2}\hat{q}_{\delta,j}^2\sqrt{det(K(z_j))}\\
&-\sum_{j=1}^m\frac{\pi\delta^2g_{z_j}(T_{z_j}z_j,T_{z_j}z_j)}{\left( \ln\frac{R}{s_{\delta,j}}\right)^2}\hat{q}_{\delta,j}^2\sqrt{det(K(z_j))}\\
&+\sum_{1\leq i\neq j\leq m}\frac{\pi\delta^2\bar{G}_{z_i}(T_{z_i}z_j,T_{z_i}z_i)}{\ln\frac{R}{s_{\delta,i}}\ln\frac{R}{s_{\delta,j}}}\hat{q}_{\delta,i}\hat{q}_{\delta,j}\sqrt{det(K(z_j))}\\
&-\sum_{j=1}^m\frac{\pi\delta^2}{2\left(  \ln\frac{R}{s_{\delta,j}}\right)^2}\hat{q}_{\delta,j}^2\sqrt{det(K(z_j))}+O\left( \frac{\delta^2}{|\ln\varepsilon|^2}\right).
\end{split}
\end{equation*}
The result follows from \eqref{2000} and the fact that $ \hat{q}_{\delta,j}=q(z_j)+O\left( \frac{1}{|\ln\varepsilon|}\right). $

\end{proof}
\textbf{Proof of Theorem \ref{thm1}:} By Propositions \ref{pro401} and \ref{order of main term}, we obtain
\begin{equation*}
P_\delta(Z)=\sum_{j=1}^m\frac{\pi\delta^2}{\ln\frac{R}{\varepsilon}}q^2\sqrt{det(K)}(z_j)+O\left( \frac{\delta^2\ln|\ln\varepsilon|}{|\ln\varepsilon|^2}\right).
\end{equation*}
If each $ x_{0,j}$ is a strict  local maximum (minimum) point of $ q^2\sqrt{det(K)} $ for every $j=1,\cdots,m$, then for $ \delta>0 $ sufficiently small, there exist at least one point $ z_{i,\delta} $ near $ x_{0,i} $ such that $ Z_\delta=(z_{1,\delta}, \cdots, z_{m,\delta}) $ is a maximum (minimum) point of $ P_\delta $ and as $ \delta\to 0 $,
\begin{equation*}
(z_{1,\delta}, \cdots, z_{m,\delta})\to (x_{0,1},\cdots,x_{0,m}).
\end{equation*}
 Hence, we get a solution $ w_\delta=V_{\delta,Z}+\omega_{\delta,Z} $ of \eqref{eq1}. Let $ u_\varepsilon=|\ln\varepsilon|w_\delta $ and $ \delta=\varepsilon|\ln\varepsilon|^{-\frac{p-1}{2}} $, we get solutions of \eqref{eq1-1}. Define the vortex set $ \bar{A}_{\varepsilon,i}=\{u_\varepsilon>q\ln\frac{1}{\varepsilon}\}\cap B_{\bar{\rho}}(x_{0,i})  $. From Lemma \ref{lemA-5}, we can find  constants $ R_1, R_2>0 $ such that
\begin{equation*}
B_{R_1\varepsilon}(z_{i,\delta})\subseteq \bar{A}_{\varepsilon,i}\subseteq B_{R_2\varepsilon}(z_{i,\delta})\subseteq B_{\bar{\rho}}(x_{0,i}).
\end{equation*}
It suffices to calculate the circulation of $ u_\varepsilon. $ Define $ \kappa_i(u_\varepsilon)=\frac{1}{\varepsilon^2}\int_{B_{\bar{\rho}}(x_{0,i})}(u_\varepsilon-q\ln\frac{1}{\varepsilon})^{p}_+dx. $ We have
\begin{lemma}\label{circulation}
	There holds
	\begin{equation*}
	\lim_{\varepsilon\to 0}\kappa_i(u_\varepsilon)=2\pi q\sqrt{det(K)}(x_{0,i}).
	\end{equation*}
\end{lemma}
\begin{proof}
	
It follows from  \eqref{q_i choice}, \eqref{203} and Proposition \ref{exist and uniq of w} that
\begin{equation*}
\begin{split}
&\frac{1}{\varepsilon^2}\int_{B_{\bar{\rho}}(x_{0,i})}(u_\varepsilon-q\ln\frac{1}{\varepsilon})^{p}_+dx\\
=&	\frac{|\ln\varepsilon|^p}{\varepsilon^2}\int_{B_{\bar{\rho}}(x_{0,i})}(w_\delta-q)^{p}_+dx\\
=&\frac{|\ln\varepsilon|^p}{\varepsilon^2}\int_{B_{Ls_{\delta,i}}(z_{i,\delta})}\left(V_{\delta, z_{i,\delta}, \hat{q}_{\delta,i}, z_{i,\delta}}(x)-\hat{q}_{\delta,i}+O\left( \frac{\ln|\ln\varepsilon|}{|\ln\varepsilon|^2}\right)\right)^{p}_+dx\\
=&\frac{|\ln\varepsilon|^p}{\varepsilon^2}s_{\delta,i}^2\left( \frac{\delta}{s_{\delta,i}}\right)^{\frac{2p}{p-1}}\int_{|T_{z_{i,\delta}}x|\leq 1}\phi(T_{z_{i,\delta}}x)^pdx+o(1)\\
=&\frac{|\ln\varepsilon|}{\delta^2}\delta^2|\phi'(1)|^{p-1}\left( \frac{\ln\frac{R}{s_{\delta,i}}}{\hat{q}_{\delta,i}}\right)^{p-1}|\phi'(1)|^{-p}\left( \frac{\ln\frac{R}{s_{\delta,i}}}{\hat{q}_{\delta,i}}\right)^{-p}\sqrt{det(K(z_{i,\delta}))}\cdot 2\pi|\phi'(1)|+o(1)\\
=& \frac{2\pi\hat{q}_{\delta,i} |\ln\varepsilon|}{\ln\frac{R}{s_{\delta,i}}}\cdot \sqrt{det(K(z_{i,\delta}))}+o(1) \\
\to& 2\pi q \sqrt{det(K)}(x_{0,i})\ \ \ \  \text{as}\ \delta\to0.
\end{split}
\end{equation*}

\end{proof}
The proof of Theorem \ref{thm1} is thus complete.

\section{Proof of Theorem \ref{thm01} and \ref{thm02}}
Consider the problem
\begin{equation}\label{eq01}
\begin{cases}
-\varepsilon^2\text{div}(K_H(x)\nabla u)=  \left(u-\left( \frac{\alpha|x|^2}{2}+\beta\right)\ln\frac{1}{\varepsilon}\right)^{p}_+,\ \ &x\in B_{R^*}(0),\\
u=0,\ \ &x\in\partial  B_{R^*}(0),
\end{cases}
\end{equation}
where $ \alpha,\beta  $ are any given constants satisfying $ \min_{x\in B_{R^*}(0)}\frac{\alpha|x|^2}{2}+\beta>0 $. Let $ v=u\backslash|\ln\varepsilon|  $ and $ \delta=\varepsilon|\ln\varepsilon|^{-\frac{p-1}{2}} $, then
\begin{equation}\label{eq02}
\begin{cases}
-\delta^2\text{div}(K_H(x)\nabla v)=  \left( v-\left( \frac{\alpha|x|^2}{2}+\beta\right) \right)^{p}_+,\ \ &x\in B_{R^*}(0),\\
v=0,\ \ &x\in\partial  B_{R^*}(0).
\end{cases}
\end{equation}
Note that \eqref{eq02} coincides with \eqref{eq1} with $ q=\frac{\alpha|x|^2}{2}+\beta $, $ K=K_H $ and $ \Omega=B_{R^*}(0) $. However, results of Theorem \ref{thm01} can not be deduced directly from those of Theorem \ref{thm1} since $ q^2\sqrt{det(K_H)} $ is a radial function and has no strict local extreme points in $ B_{R^*}(0) $. In this case one can also use the reduction procedure to construct solutions of \eqref{eq02}, by using the rotational symmetry of $ K_H, q $ and the domain $ B_{R^*}(0). $

Let $ h(r)=h(|x|)=q^2\sqrt{det(K_H)}(x) $ for any $ x\in B_{R^*}(0) $. We call $ z^* $ is a  $ strict $ $ local $ $ maximum $ $ (minimum) $ $ point $ of  $q^2\sqrt{det(K_H)}$  $ up $ $ to $ $ ratation $ in $ B_{R^*}(0) $, if $ |z^*| $ is a strict local maximum (minimum) point of $ h $ in $ (0,{R^*}) $.

Let $ z_1  $ be a strict local maximum (minimum) point of $ q^2\sqrt{det(K_H)} $ up to rotation. Define $ \mathcal{N}=B_{\bar{\rho}}(z_1) $. Then we can  construct solutions of \eqref{eq02} being of the form $ v_\delta=V_{\delta,z}+\omega_{\delta,z} $, where $ z $ is near $ z_1. $

Indeed, by Lemma \ref{coercive esti} and Proposition \ref{exist and uniq of w}, for any $ z\in \mathcal{N} $ and $ \delta $ sufficiently small there exists a unique $ \omega_{\delta,z}\in E_{\delta,z} $ such that $ Q_\delta L_\delta \omega_{\delta,z}=Q_\delta l_\delta+Q_\delta R_\delta(\omega_{\delta,z}). $ So the final step is to prove the existence of $ z=z_{\delta} $ near $ z_1 $ satisfying $ \nabla_zP_\delta(z_\delta)=0. $ We claim that
\begin{equation}\label{501}
P_\delta(z)=\frac{\pi\delta^2}{\ln\frac{R}{\varepsilon}}q^2\sqrt{det(K_H)}(z)+N_\delta(z),
\end{equation}
where $ N_\delta(z) $ is a $ O(\frac{\delta^2\ln|\ln\varepsilon|}{|\ln\varepsilon|^2}) -$perturbation term which is invariant under rotation. In fact, we can choose $ T_x^{-1}=\begin{pmatrix}
\cos\theta_x & -\sin\theta_x  \\
\sin\theta_x &\cos\theta_x
\end{pmatrix}\begin{pmatrix}
\frac{k}{\sqrt{k^2+|x|^2}} & 0  \\
0 &1
\end{pmatrix} $ in \eqref{T_z choice}, where $ (|x|,\theta_x) $ is the polar coordinate of $ x $. By the rotational symmetry of $ q $ and  the domain $ B_R(0) $, one can  prove that for every $ \theta\in[0,2\pi] $,  if we define $ \bar{z}=\bar{R}_\theta(z) $, then
\begin{equation*}
V_{\delta,\bar{z}}(\bar{R}_\theta(x))=V_{\delta,z}(x)\ \text{and} \ \omega_{\delta,\bar{z}}(\bar{R}_\theta(x))=\omega_{\delta,z}(x)    \ \ \text{for any}\  x\in B_R(0).
\end{equation*}
Hence one computes directly  that $ P_\delta(\bar{z})=P_\delta(z) $, i.e., $  P_\delta $ is a radially symmetric function. Note that $ q^2\sqrt{det(K_H)}(x)=\left( \frac{\alpha|x|^2}{2}+\beta\right)^2\cdot \frac{k}{\sqrt{k^2+|x|^2}} $ is also radially symmetric. Thus by  Proposition \ref{pro401} and \ref{order of main term},  we have \eqref{501}.

Since $ z_1 $ is a strict local maximum (minimum) point of $ q^2\sqrt{det(K_H)} $ up to rotation, it is not hard to prove  the existence of $ z_{\delta} $ near $ z_1 $ satisfying $ \nabla_zP_\delta(z_\delta)=0 $, which yields a solution $ v_\delta $ of \eqref{eq02}. Let $ u_\varepsilon=v_\delta|\ln\varepsilon| $, then $ u_\varepsilon $ is a solution of \eqref{eq01}. Moreover, by Lemma \ref{circulation}, one has
\begin{equation*}
\lim_{\varepsilon\to 0}\frac{1}{\varepsilon^2}\int_{B_{\bar{\rho}}(z_1)}\left( u_\varepsilon-q\ln\frac{1}{\varepsilon}\right)^{p}_+dx=2\pi q(z_1)\sqrt{det(K_H(z_1))}=\frac{k\pi(\alpha|z_1|^2+2\beta)}{\sqrt{k^2+|z_1|^2}}.
\end{equation*}
To conclude, we have
\begin{theorem}\label{thmA}
Let $ \alpha,\beta $ be two constants satisfying $ \min_{x\in B_{R^*}(0)}\left( \frac{\alpha|x|^2}{2}+\beta\right) >0 $ and $ z_1\in B_{R^*}(0)$ be a strict local maximum (minimum) point of $ \left( \frac{\alpha|x|^2}{2}+\beta\right)^2\cdot \frac{k}{\sqrt{k^2+|x|^2}} $ up to rotation. Then there exists $ \varepsilon_0>0 $, such that for any $ \varepsilon\in(0,\varepsilon_0] $, \eqref{eq01} has a solution $ u_\varepsilon $ satisfying the following properties:
\begin{enumerate}
	\item  Define $ A_\varepsilon=\left\{u_\varepsilon>\left( \frac{\alpha|x|^2}{2}+\beta\right)\ln\frac{1}{\varepsilon}\right\} $. Then $ \lim\limits_{\varepsilon\to 0}dist(A_\varepsilon, z_1)=0 $.
	\item $ \lim\limits_{\varepsilon\to 0}\frac{1}{\varepsilon^2}\int_{A_\varepsilon}\left( u_\varepsilon-\left( \frac{\alpha|x|^2}{2}+\beta\right)\ln\frac{1}{\varepsilon}\right)^p_+dx=\frac{k\pi(\alpha|z_1|^2+2\beta)}{\sqrt{k^2+|z_1|^2}}. $
	\item There exist  $ R_1,R_2>0 $ satisfying
	\begin{equation*}
	R_1\varepsilon\leq \text{diam}(A_\varepsilon)\leq R_2\varepsilon.
	\end{equation*}
	
\end{enumerate}
\end{theorem}

\bigskip

\textbf{Proof of Theorem \ref{thm01}}:
To prove Theorem \ref{thm01}, we define for every $ r_*\in (0,R^*) $, $ c>0 $ and
\begin{equation}\label{502}
\alpha=\frac{c}{4\pi k\sqrt{k^2+r_*^2}},\ \ \beta=\frac{\alpha}{2}(3r_*^2+4k^2).
\end{equation}
One computes directly that $ (r_*,0)$ is a strict minimum point of $ q^2\sqrt{det(K_H)}(x)=\left( \frac{\alpha|x|^2}{2}+\beta\right)^2\cdot \frac{k}{\sqrt{k^2+|x|^2}} $ up to rotation and that
\begin{equation*}
2\pi q\sqrt{det(K_H)}((r_*,0))=\frac{k\pi(\alpha r_*^2+2\beta)}{\sqrt{k^2+r_*^2}}=c.
\end{equation*}
Hence by Theorem \ref{thmA}, for any $ \varepsilon $ small there exists a solution $ u_\varepsilon $ of \eqref{eq01} concentrating near $ (r_*,0) $ with $ \alpha=\frac{c}{4\pi k\sqrt{k^2+r_*^2}} $ and $\frac{1}{\varepsilon^2} \int_{A_\varepsilon}\left( u_\varepsilon-\left( \frac{\alpha|x|^2}{2}+\beta\right)\ln\frac{1}{\varepsilon}\right)^p_+dx $ tending to $ c $. Define for any $ (x_1,x_2,x_3)\in B_{R^*}(0)\times\mathbb{R}, t\in\mathbb{R}  $
\begin{equation*}
\mathbf{w}_\varepsilon(x_1,x_2,x_3,t)= \frac{w_\varepsilon(x_1,x_2,x_3,t)}{k}\overrightarrow{\zeta},
\end{equation*}
where $ w_\varepsilon(x_1,x_2,x_3,t) $ is a helical function satisfying
\begin{equation}\label{503}
w_\varepsilon(x_1,x_2,0,t)=\frac{1}{\varepsilon^2}\left( u_\varepsilon(\bar{R}_{-\alpha|\ln\varepsilon| t}(x_1,x_2))-\left( \frac{\alpha|(x_1,x_2)|^2}{2}+\beta\right)\ln\frac{1}{\varepsilon}\right)^p_+.
\end{equation}
Direct computations show that $ w_\varepsilon(x_1,x_2,0,t) $ satisfies \eqref{vor str eq} and $ \mathbf{w}_\varepsilon $ is a left-handed helical vorticity field of \eqref{Euler eq2}. Moreover, $ w_\varepsilon(x_1,x_2,0,t) $ rotates clockwise around the origin with angular velocity $ \alpha|\ln\varepsilon| $. By \eqref{1001}, the circulation of $ \mathbf{w}_\varepsilon $ satisfies
\begin{equation*}
\iint_{A_\varepsilon}\mathbf{w}_\varepsilon\cdot\mathbf{n}d\sigma=\frac{1}{\varepsilon^2} \int_{A_\varepsilon}\left( u_\varepsilon-\left( \frac{\alpha|x|^2}{2}+\beta\right)\ln\frac{1}{\varepsilon}\right)^p_+dx\to c,\ \ \text{as}\ \varepsilon\to0.
\end{equation*}

It suffices to prove that the vorticity field $ \mathbf{w}_\varepsilon $ tends asymptotically to \eqref{1007} in sense of \eqref{1005}.  Define $ P(\tau) $ the intersection point of the curve  parameterized by \eqref{1007} and the $ x_1Ox_2 $ plane. Note that the helix \eqref{1007} corresponds uniquely to the motion of $ P(\tau) $. Taking $ s=\frac{b_1\tau}{k} $ into \eqref{1007}, one computes directly that  $ P(\tau) $ satisfies a 2D point vortex model
\begin{equation*}
\begin{split}
P(\tau)=\bar{R}_{\alpha'\tau}((r_*,0)),
\end{split}
\end{equation*}
where
\begin{equation*}
\alpha'=\frac{1}{\sqrt{k^2+r_*^2}}\left( a_1+\frac{b_1}{k}\right)=\frac{c}{4\pi k\sqrt{k^2+r_*^2}},
\end{equation*}
which is equal to $ \alpha $ in \eqref{502}. Thus by the construction, we  readily check that the support set of $ w_\varepsilon(x_1,x_2,0,|\ln\varepsilon|^{-1}\tau) $ defined by  \eqref{503} concentrates near $ P(\tau) $ as $ \varepsilon\to 0 $, which implies that, the vorticity field $ \mathbf{w}_\varepsilon $ tends asymptotically to \eqref{1007} in sense of \eqref{1005}.
The proof of Theorem \ref{thm01} is thus complete.

\textbf{Proof of Theorem \ref{thm02}}: The proof of Theorem \ref{thm02} is similar to that of Theorem \ref{thm01}, by constructing multiple concentration solutions $ u_\varepsilon $ of \eqref{eq01} (or equivalently, $ v_\delta $ of \eqref{eq02}) with polygonal symmetry. Indeed, let $ m\geq2 $ be an integer, $ r_*\in (0,R) $, $ c>0 $ and $ \alpha,\beta $ satisfying \eqref{502}. For any $ z_1\in B_{\bar{\rho}}((r_*,0)) $, define $ z_i=\bar{R}_{\frac{2(i-1)\pi}{m}}(z_1) $ for $ i=2,\cdots,m. $ Our goal is to construct solutions of \eqref{eq02} being of the form $  \sum_{i=1}^mV_{\delta,z_i}+\omega_{\delta} $. Using the symmetry of $ K_H, q $ and $ B_{R^*}(0) $, one can  prove that
\begin{equation*}
P_\delta(z_1,\cdots,z_m)=\frac{m\pi\delta^2}{\ln\frac{R}{\varepsilon}}q^2\sqrt{det(K_H)}(z_1)+N_\delta(z_1),
\end{equation*}
where $ N_\delta(z_1) $ is a $ O\left( \frac{\delta^2\ln|\ln\varepsilon|}{|\ln\varepsilon|^2}\right)  -$perturbation term which is invariant under rotation. The rest of the proof is exactly the same as
in that of Theorem \ref{thm01} and we omit the details.

\section{Appendix: Some basic estimates}
In this Appendix, we give some results  which have been repeatedly used before.

First, we give estimates of the Green's function. Let $ G(x,y) $ be the Green's function of $ -\Delta $ in $ \Omega $ with zero-Dirichlet boundary condition and $ h(x,y):=\frac{1}{2\pi}\ln \frac{1}{|x-y|}-G(x,y)$ be the regular part of
$ G(x,y) $, then
\begin{customthm}{A.1}[Lemma 4.1, \cite{De2}]\label{lemA-1}
For all $x,y \in \Omega$, there hold
\begin{equation}\label{B-1}
\begin{split}
h(x,y) \leq \frac{1}{2\pi}\ln  \frac{1}{\max \{ |x-y|, dist(x,\partial \Omega),dist(y,\partial \Omega) \} },
\end{split}
\end{equation}
\begin{equation}\label{B-2}
\begin{split}
h(x,y) \ge \frac{1}{2\pi}\ln\frac{1}{|x-y|+2\max \{ dist(x,\partial \Omega),dist(y,\partial \Omega) \} }.
\end{split}
\end{equation}
	
\end{customthm}
Note that the Green's function $ G(x,y) $ is only determined by the domain. To get \eqref{200-1} and \eqref{200}, one must get estimates of $ \frac{\partial h_{z_i}(T_{z_i}z_i, T_{z_i}z_i)}{\partial z_{i,\hbar}} $ and $ \frac{\partial G_{z_i}(T_{z_i}z_i, T_{z_i}z_j)}{\partial z_{i,\hbar}} $, which involve the $ C^1- $dependence of the Green's function on the domain. The following $ Hadamard $ $ variational $ $ formula $ gives  qualitative estimates of the first derivative of the Green's function on the domain, see \cite{Had, KU, Schi} for instance.

More precisely, let $ \Omega $ be a simply-connected bounded domain with smooth boundary and $ \Omega_\varepsilon $ be the perturbation of the domain $ \Omega $ whose boundary $ \partial \Omega_\varepsilon $ is expressed in such a way that
\begin{equation}\label{perturbation of domain}
\partial \Omega_\varepsilon=\{x+\varepsilon\rho(x)\mathbf{n}_x \mid x\in \partial \Omega\},
\end{equation}
where $ \rho\in C^{\infty}(\partial \Omega) $ and $ \mathbf{n}_x $ is
the unit outer normal to $ \partial \Omega $. 	For all $y,z \in \Omega$, we define
\begin{equation*}
\mathcal{\delta}_\rho G(y,z):=\lim_{\varepsilon\to 0}\varepsilon^{-1}(G_\varepsilon(y,z)-G(y,z)),
\end{equation*}
\begin{equation*}
\delta_\rho h(y,z):=\lim_{\varepsilon\to 0}\varepsilon^{-1}(h_\varepsilon(y,z)-h(y,z)),
\end{equation*}
where $ G_\varepsilon(y,z)=-\frac{1}{2\pi}\ln|y-z|-h_\varepsilon(y,z) $ is the Green's function of $ -\Delta $ in $ \Omega_\varepsilon $ with Dirichlet  boundary condition and $ h_\varepsilon(y,z) $ is the regular part of $ G_\varepsilon(y,z) $. We have
\begin{customthm}{A.2}[\cite{KU}]\label{lemA-6}
It holds that
\begin{equation}\label{Hadamard Varia}
\delta_\rho G(y,z)=-\delta_\rho h(y,z)=\int_{\partial \Omega}\frac{\partial G(y,x)}{\partial \mathbf{n}_x}\frac{\partial G(z,x)}{\partial \mathbf{n}_x}\rho(x)d\sigma_x,
\end{equation}
where $ d\sigma_x $	denotes the surface element of $ \partial \Omega. $
\end{customthm}

Let $ h_{\hat{x}}(x, y) $ be the regular part of Green's function $ G_{\hat{x}}(x, y) $ of $ -\Delta $ in $ T_{\hat{x}}(\Omega) $, where $ T \in C^\infty $ satisfies $ (T_{\hat{x}}^{-1})(T_{\hat{x}}^{-1})^t=K(\hat{x})$ for any $ \hat{x}\in\Omega $. Based on Lemmas \ref{lemA-1} and \ref{lemA-6}, one can get the following estimates.
\begin{customthm}{A.3}\label{lemA-2}
 For any $ Z\in \mathcal{M}$   and $ \delta $ sufficiently small,  there exists $ C>0 $ independent of $ \varepsilon $ and $ Z $ satisfying
\begin{equation}\label{A201}
\begin{split}
|h_{z_i}(T_{z_i}z_i, T_{z_i}z_j)|+|\nabla h_{z_i}(T_{z_i}z_i, T_{z_i}z_j)|\le C, \ \  1\leq i,j\leq m,
\end{split}
\end{equation}
\begin{equation}\label{A202}
\begin{split}
|G_{z_i}(T_{z_i}z_i, T_{z_i}z_j)|+|\nabla G_{z_i}(T_{z_i}z_i, T_{z_i}z_j)|\le C, \ \ 1\leq i\neq j\leq m,
\end{split}
\end{equation}
\begin{equation}\label{A203}
\bigg|\frac{\partial h_{z_i}(T_{z_i}z_i, T_{z_i}z_j)}{\partial z_{i,\hbar}}\bigg|\leq C,  \ \  1\leq i,j\leq m, \hbar=1,2,
\end{equation}
and
\begin{equation}\label{A204}
\bigg|\frac{\partial G_{z_i}(T_{z_i}z_i, T_{z_i}z_j)}{\partial z_{i,\hbar}}\bigg|\leq C,  \ \  1\leq i\neq j\leq m, \hbar=1,2.
\end{equation}
	
\end{customthm}

\begin{proof}
Since for any $ Z\in \mathcal{M} $ all the eigenvalues of $ T_{z_i} $ have positive upper  and lower bounds uniformly about $ Z $, we can find $ C_1>0 $ such that $ dist(T_{z_i}z_i, T_{z_i}(\Omega))\geq C_1>0. $ Hence by Lemma \ref{lemA-1}, we have $ |h_{z_i}(T_{z_i}z_i, T_{z_i}z_j)|\leq C. $ By the interior gradient estimates for the harmonic functions, we get  $ |\nabla h_{z_i}(T_{z_i}z_i, T_{z_i}z_j)|\le C. $ So \eqref{A201} holds. \eqref{A202} follows from \eqref{A201} and the definition of $ G_{z_i} $. Indeed, using the interior gradient estimates, one can also get that for any integer $ l\geq 1 $ and $ Z\in \mathcal{M} $,
\begin{equation*}
|\nabla^l h_{z_i}(T_{z_i}z_i, T_{z_i}z_j)|\le C.
\end{equation*}

Now we estimate $ \frac{\partial h_{z_i}(T_{z_i}z_i, T_{z_i}z_j)}{\partial z_{i,\hbar}} $. For $ 1\leq i\neq j\leq m $, we have
\begin{equation*}
\begin{split}
\frac{\partial h_{z_i}(T_{z_i}z_i, T_{z_i}z_j)}{\partial z_{i,\hbar}}=&\lim_{\varepsilon\to 0}\varepsilon^{-1}\left(h_{z_i+\varepsilon\mathbf{e}_\hbar}(T_{z_i+\varepsilon\mathbf{e}_\hbar}(z_i+\varepsilon\mathbf{e}_\hbar), T_{z_i+\varepsilon\mathbf{e}_\hbar}z_j) -h_{z_i}(T_{z_i}z_i, T_{z_i}z_j) \right) \\
=&\lim_{\varepsilon\to 0}\varepsilon^{-1}\left(h_{z_i+\varepsilon\mathbf{e}_\hbar}(T_{z_i+\varepsilon\mathbf{e}_\hbar}(z_i+\varepsilon\mathbf{e}_\hbar), T_{z_i+\varepsilon\mathbf{e}_\hbar}z_j) -h_{z_i+\varepsilon\mathbf{e}_\hbar}(T_{z_i}z_i, T_{z_i}z_j) \right)\\
&+\lim_{\varepsilon\to 0}\varepsilon^{-1}\left(h_{z_i+\varepsilon\mathbf{e}_\hbar}(T_{z_i}z_i, T_{z_i}z_j) -h_{z_i}(T_{z_i}z_i, T_{z_i}z_j) \right),
\end{split}
\end{equation*}
where $ \mathbf{e}_\hbar $ is the unit vector of $ x_\hbar-$axis. By \eqref{A201},
\begin{equation*}
\big|\lim_{\varepsilon\to 0}\varepsilon^{-1}\left(h_{z_i+\varepsilon\mathbf{e}_\hbar}(T_{z_i+\varepsilon\mathbf{e}_\hbar}(z_i+\varepsilon\mathbf{e}_\hbar), T_{z_i+\varepsilon\mathbf{e}_\hbar}z_j) -h_{z_i+\varepsilon\mathbf{e}_\hbar}(T_{z_i}z_i, T_{z_i}z_j) \right)\big|\leq C.
\end{equation*}
Note that $ T_{z_i+\varepsilon\mathbf{e}_\hbar}(\Omega) $ is a perturbation of $ T_{z_i}(\Omega) $.
Since $ T  $ is a $ C^{\infty} $ matrix-valued function and $ \partial \Omega $ is a smooth curve,  we find that $ \rho_{i,\hbar}(x): T_{z_i}(\Omega)\to \mathbb{R} $,  the normal displacement function defined by \eqref{perturbation of domain}, is smooth about $ x\in T_{z_i}(\Omega) $. Hence by Lemmas \ref{lemA-1} and \ref{lemA-6}, for any  $ Z\in \mathcal{M} $
\begin{equation*}
\begin{split}
&\big|\lim_{\varepsilon\to 0}\varepsilon^{-1}\left(h_{z_i+\varepsilon\mathbf{e}_\hbar}(T_{z_i}z_i, T_{z_i}z_j) -h_{z_i}(T_{z_i}z_i, T_{z_i}z_j) \right)\big|\\
=&\bigg| \int_{\partial (T_{z_i}(\Omega))}\frac{\partial G_{z_i}(T_{z_i}z_i,x)}{\partial \mathbf{n}_x}\frac{\partial G_{z_i}(T_{z_i}z_j,x)}{\partial \mathbf{n}_x}\rho_{i,\hbar}(x)d\sigma_x\bigg| \\
\leq& C\max_{ \partial (T_{z_i}(\Omega)) }|\nabla G_{z_i}(T_{z_i}z_i,\cdot)| \max_{ \partial (T_{z_i}(\Omega)) }|\nabla G_{z_i}(T_{z_i}z_j,\cdot)| \max_{ \partial (T_{z_i}(\Omega)) }|\rho_{i,\hbar}(\cdot)|\leq C.
\end{split}
\end{equation*}
So we get $ \big|\frac{\partial h_{z_i}(T_{z_i}z_i, T_{z_i}z_j)}{\partial z_{i,\hbar}}\big|\leq C $. \eqref{A204} can be proved similarly. Indeed,  one can also get that for any integer $ l\geq 1 $ and $ Z\in \mathcal{M} $,
\begin{equation*}
\bigg|\nabla^l \frac{\partial h_{z_i}(T_{z_i}z_i, T_{z_i}z_j)}{\partial z_{i,h}}\bigg|\le C.
\end{equation*}
\end{proof}

The following lemma shows the existence of $ \hat{q}_i $ satisfying \eqref{q_i choice}.
\begin{customthm}{A.4}\label{lemA-3}
For any $ Z\in \mathcal{M} $ and $ \delta $ sufficiently small,  there exist $ \hat{q}_i=\hat{q}_{\delta,i}(Z) $ satisfying
\begin{equation}\label{A301}
\hat{q}_i=q(z_i)+\frac{\hat{q}_i}{\ln\frac{R}{\varepsilon}}g_{z_i}(T_{z_i}z_i, T_{z_i}z_i)-\Sigma_{j\neq i}\frac{\hat{q}_j}{\ln\frac{R}{\varepsilon}}\bar{G}_{z_j}(T_{z_j}z_i, T_{z_j}z_j),
\end{equation}
where $ \bar{G}_{z_j}(x,y)=\ln\frac{ R}{|x-y|}-g_{z_j}(x,y)=2\pi G_{z_j}(x,y) $ for any $ x,y\in T_{z_j}(\Omega) $.

\end{customthm}
\begin{proof}
By Lemma \ref{lemA-2}, it is not hard to prove that there exists the unique $ \hat{q}_i $ satisfying \eqref{A301}. Moreover, we have $ \hat{q}_i=q(z_i)+O(\frac{1}{|\ln\varepsilon|}) $ and $ \frac{\partial \hat{q}_i}{\partial z_{i,h}}=O(1) $ for $ h=1,2. $
\end{proof}

This lemma shows the well-definedness of $ Q_\delta $ defined by \eqref{206}.
\begin{customthm}{A.5}\label{lemA-4}
$ Q_\delta $ is well-defined. Moreover, for any $ q\in[1,+\infty) $, $ u\in L^q(\Omega) $ with $ supp(u)\subseteqq \cup_{j=1}^mB_{Ls_{\delta,j}}(z_j)$ for some $ L>1 $, there holds for some $ C>0 $
\begin{equation*}
||Q_\delta u||_{L^q}\leq C||u||_{L^q}.
\end{equation*}
	
\end{customthm}
\begin{proof}
By \eqref{coef of C}, we already know that there exists a unique $ C_{j,h} $ satisfying \eqref{207}.

By the assumption, for any $ q\in[1,+\infty) $, $ l=1,\cdots,m, h=1,2, $ one has
\begin{equation*}
\begin{split}
C_{l,h}=&O\left( |\ln\varepsilon|^{p+1}\sum_{i=1}^m\sum_{\hbar=1}^2\int_{\cup_{j=1}^mB_{Ls_{\delta,j}}(z_j)}u\frac{\partial V_{\delta,Z,i}}{\partial z_{i,\hbar}}\right) \\
=&O\left( |\ln\varepsilon|^{p+1}\sum_{i=1}^m\sum_{\hbar=1}^2||u||_{L^q(\Omega)}||\frac{\partial V_{\delta,Z,i}}{\partial z_{i,\hbar}}||_{L^{q'}(\cup_{j=1}^mB_{Ls_{\delta,j}}(z_j))}\right) \\
=&O\left( |\ln\varepsilon|^{p+1}||u||_{L^q(\Omega)}\sum_{i=1}^m\sum_{\hbar=1}^2\frac{s_{\delta,j}^{\frac{2}{q'}-1}}{|\ln\varepsilon|}\right) \\
=&O(|\ln\varepsilon|^{p}\varepsilon^{1-\frac{2}{q}})||u||_{L^q(\Omega)},
\end{split}
\end{equation*}
where $ q' $ denotes the conjugate exponent of $ q $. Hence we get
\begin{equation*}
\begin{split}
&\sum_{j=1}^m\sum_{h=1}^2C_{j,h}\frac{\partial}{\partial z_{j,h}}(-\delta^2\text{div}(K(z_j)\nabla   V_{\delta, Z,j}))\\
=&\sum_{j=1}^m\sum_{h=1}^2C_{j,h}(V_{\delta, z_{j}, \hat{q}_{\delta,j}, z_{j}}-\hat{q}_{\delta,j})^{p-1}_+\left( \frac{\partial V_{\delta, z_{j}, \hat{q}_{\delta,j}, z_{j}}}{\partial z_{j,h}}-\frac{\partial \hat{q}_{\delta,j}}{\partial z_{j,h}}\right) \\
=&O\left( \sum_{j=1}^m\sum_{h=1}^2C_{j,h}|\ln\varepsilon|^{-(p-1)}\cdot \frac{s_{\delta,j}^{\frac{2}{q}-1}}{|\ln\varepsilon|}\right) \\
=&O(1)||u||_{L^q(\Omega)}\ \ \ \ \text{in}\ L^q(\Omega).
\end{split}
\end{equation*}
Thus $ ||Q_\delta u||_{L^q(\Omega)}\leq C||u||_{L^q(\Omega)} $.

\end{proof}

\begin{customthm}{A.6}\label{lemA-5}
There exists a constant $ L > 1 $ such that for $ \varepsilon $ small
\begin{equation*}
V_{\delta,Z}-q>0,\ \ \  \ \text{in}\ \   \cup_{j=1}^m\left( T_{z_j}^{-1}B_{\left( 1-L\frac{\ln|\ln\varepsilon|}{|\ln\varepsilon|}\right) s_{\delta,j}}(0)+z_j\right),
\end{equation*}
\begin{equation*}
V_{\delta,Z}-q<0,\ \ \  \ \text{in}\ \  \Omega\backslash\cup_{j=1}^m\left( T_{z_j}^{-1}B_{Ls_{\delta,j}}(0)+z_j\right).
\end{equation*}
	
\end{customthm}

\begin{proof}
The proof is similar to that of Lemma A.1 in \cite{CLW}. If $ |T_{z_j}(x-z_j)|\leq \left( 1-L\frac{\ln|\ln\varepsilon|}{|\ln\varepsilon|}\right) s_{\delta,j}   $, then by \eqref{203} and $ \phi'(1)<0 $ we have
\begin{equation*}
\begin{split}
V_{\delta,Z}-q(x)=&V_{\delta,z_j,\hat{q}_{\delta,j}, z_j}(x)-\hat{q}_{\delta,j}+O\left( \frac{\ln|\ln\varepsilon|}{|\ln\varepsilon|^2}\right) \\
= &\frac{\hat{q}_{\delta,j}}{|\phi'(1)|\ln\frac{R}{s_{\delta,j}}}\phi\left( \frac{|T_{z_j}(x-z_j)|}{s_{\delta,j}}\right) +O\left( \frac{\ln|\ln\varepsilon|}{|\ln\varepsilon|^2}\right) >0,
\end{split}
\end{equation*}
if $ L $ is sufficiently large.

On the other hand, if $ \tau>0 $ small and $ |T_{z_j}(x-z_j)|\geq s_{\delta,j}^\tau $ for any $ j=1,\cdots,m $, then
\begin{equation*}
\begin{split}
V_{\delta,Z}-q(x)=&\sum_{j=1}^m\left( V_{\delta,z_j,\hat{q}_{\delta,j}, z_j}(x)-\frac{\hat{q}_{\delta,j}}{\ln\frac{R}{s_{\delta,j}}}g_{z_j}(T_{z_j}x,T_{z_j}z_j)\right) -q(x)\\
\leq &\sum_{j=1}^m\frac{\hat{q}_{\delta,j}\ln\frac{R}{s_{\delta,j}^\tau}}{\ln\frac{R}{s_{\delta,j}}}-C\\
\leq &\tau \sum_{j=1}^m \hat{q}_{\delta,j}-C<0.
\end{split}
\end{equation*}
If $ Ls_{\delta,j}\leq |T_{z_j}(x-z_j)|\leq s_{\delta,j}^\tau $, then by Lemma \ref{lemA-2}, we have
\begin{equation*}
\begin{split}
&V_{\delta,Z}-q(x)\\
=&V_{\delta,z_j,\hat{q}_{\delta,j}, z_j}(x)-\frac{\hat{q}_{\delta,j}}{\ln\frac{R}{s_{\delta,j}}}g_{z_j}(T_{z_j}x,T_{z_j}z_j)-q(x)+\sum_{i\neq j}\left(\frac{\hat{q}_{\delta,i}}{\ln\frac{R}{s_{\delta,i}}}\bar{G}_{z_i}(T_{z_i}x,T_{z_i}z_i)\right)\\
=&V_{\delta,z_j,\hat{q}_{\delta,j}, z_j}(x)-q(z_j)-\frac{\hat{q}_{\delta,j}}{\ln\frac{R}{s_{\delta,j}}}g_{z_j}(T_{z_j}z_j,T_{z_j}z_j)+\sum_{i\neq j}\left(\frac{\hat{q}_{\delta,i}}{\ln\frac{R}{s_{\delta,i}}}\bar{G}_{z_i}(T_{z_i}z_j,T_{z_i}z_i)\right)+O(\varepsilon^\tau)\\
=&V_{\delta,z_j,\hat{q}_{\delta,j}, z_j}(x)-q(z_j)-\frac{\hat{q}_{\delta,j}}{\ln\frac{R}{\varepsilon}}g_{z_j}(T_{z_j}z_j,T_{z_j}z_j)+\sum_{i\neq j}\left(\frac{\hat{q}_{\delta,i}}{\ln\frac{R}{\varepsilon}}\bar{G}_{z_i}(T_{z_i}z_j,T_{z_i}z_i)\right)+O\left( \frac{\ln|\ln\varepsilon|}{|\ln\varepsilon|^2}\right) \\
=&V_{\delta,z_j,\hat{q}_{\delta,j}, z_j}(x)-\hat{q}_{\delta,j}+O\left( \frac{\ln|\ln\varepsilon|}{|\ln\varepsilon|^2}\right) \\
\leq &-\frac{\hat{q}_{\delta,j}\ln L}{\ln\frac{R}{s_{\delta,j}}}+O\left( \frac{\ln|\ln\varepsilon|}{|\ln\varepsilon|^2}\right) <0,
\end{split}
\end{equation*}
if we choose $ L $ sufficiently large.

\end{proof}

\subsection*{Acknowledgments:}

\par
D. Cao was supported by NNSF of China (grant No. 11831009). J. Wan was supported by NNSF of China (grant No. 12101045) and  Beijing
Institute of Technology Research Fund Program for Young Scholars (No.3170011182016).

\end{document}